\documentclass[12pt]{amsart}
\usepackage{geometry}
\usepackage{amsmath}
\usepackage{amsthm}
\usepackage{bbm}
\usepackage{amsfonts,amssymb,bm}
\usepackage{fancyhdr}
\usepackage{mathrsfs}
\usepackage[T1]{fontenc}
\usepackage{lmodern}
\usepackage[normalem]{ulem}
\usepackage{booktabs}
\useunder{\uline}{\ul}{}
\usepackage{mathtools}
\usepackage{appendix}
\usepackage{tikz,tikz-cd}
\usepackage{graphicx}
\usepackage[all]{xy}

\usepackage{float} 
\usepackage{amsthm}

\newcommand{\hvol}{{\widehat{\rm vol}}}
\newtheorem{thm}{Theorem}[section]
\newtheorem{defi}[thm]{Definition}
\newtheorem{lem}[thm]{Lemma}
\newtheorem{cor}[thm]{Corollary}
\newtheorem{prop}[thm]{Proposition}

\newtheorem{ex}[thm]{Example}

\newtheorem{rmk}[thm]{Remark}

\usepackage{times}
\usepackage{hyperref}
\def\ord{{\rm ord}}
\def\min{{\rm min}}
\def\max{{\rm max}}
\def\inf{{\rm inf}}
\def\sup{{\rm sup}}
\def\lim{{\rm lim}}
\def\limsup{{\rm lim\,sup}}

\def\dif{{\rm d}}

\def\Supp{{\rm Supp}}

\def\bl{{\rm Bl}}

\def\log{{\rm log}}
\def\vol{{\rm vol}}

\def\ric{{\rm Ric}}

\def\ord{{\rm ord}}

\def\fs{{\rm FS}}

\def\sym{{\rm Sym}}

\newcommand{\IC}{{\mathbb C}}

\newcommand{\IN}{{\mathbb N}}
 
\newcommand{\IP}{{\mathbb P}} 
\newcommand{\IQ}{{\mathbb Q}} 
\newcommand{\IR}{{\mathbb R}}

\newcommand{\IT}{{\mathbb T}}

\newcommand{\IZ}{{\mathbb Z}}

\newcommand{\CE}{{\mathcal E}}

\newcommand{\CI}{{\mathcal I}}
 
\newcommand{\CK}{{\mathcal K}}
\newcommand{\CL}{{\mathcal L}}

\newcommand{\CO}{{\mathcal O}}

\newcommand{\D}{\Delta}

\newcommand{\Ric}{\mathrm{Ric}}
\newcommand{\sddb}{\sqrt{-1}\partial\bar{\partial}}
\newcommand{\vphi}{\varphi}

\newcommand{\Addresses}{{
		\footnotesize
        
  	    Chi Li \par\nopagebreak
        \textsc{Department of Mathematics, Rutgers University, Piscataway, NJ, 08854, U.S.}\par\nopagebreak
        \textsc{Department of Mathematics,
        Johns Hopkins University, 
        Baltimore, MD, 21218, U.S.  }\par\nopagebreak
         \textit{E-mail address:} 
         {chi.li@rutgers.edu, cli261@jh.edu}

         \medskip

    Minghao Miao \par\nopagebreak
          \textsc{School of Mathematics, Nanjing University, Nanjing 210093, China}\par\nopagebreak
          \textit{E-mail address}:  minghao.miao@smail.nju.edu.cn 

    \medskip

    Kewei Zhang 
    \par\nopagebreak
     \textsc{School of Mathematical Sciences, Beijing Normal University, Beijing 100871, China}\par\nopagebreak
     \textit{E-mail address}:
       {kwzhang@bnu.edu.cn}

}}

\begin{document}

\title{The sharp volume gap for K\"ahler manifolds with positive Ricci curvature}

\author{Chi Li, Minghao Miao, Kewei Zhang}
\date{}
\maketitle

\begin{abstract}
   We prove a sharp volume gap estimate: if an $n$-dimensional compact K\"{a}hler manifold $(X, \omega)$ satisfies $\Ric(\omega)\ge (n+1)\omega$ and $X\not\cong \IP^n$, then $\vol(X, \omega)\le \frac{2n^n}{(n+1)^n}\vol(\IP^n,\omega_{\fs})=\frac{2^{n+1}\pi^n n^n}{(n+1)^n}$. Moreover $\vol(X, \omega)= \frac{2^{n+1}\pi^n n^n}{(n+1)^n}$ occurs if and only if $(X, \omega)$ is biholomorphically isometric to the K\"{a}hler-Einstein metric on the quadric hypersurface $Q^n$ or on the product $\IP^1\times \IP^{n-1}$. We also obtain sharp volume gap estimates for K-semistable toric log Fano pairs. 

\end{abstract}

\section{Introduction}

In Riemannian geometry the classical almost rigidity theorem for the sphere, going back to Perelman \cite{P94} and Cheeger--Colding \cite[Theorem A.1.10]{CC97}, says the following.

\begin{thm}
\label{thm:sphere}
    There exists a dimensional constant $\delta(m)>0$ such that the following holds. Let $(M^m,g)$ be a complete Riemannian manifold with $\Ric(g)\geq (m-1) g$. Assume that $$\vol(M,g)>(1-\delta(m))\vol(S^m),$$ then $M$ is diffeomorphic to $S^m$. Here $S^m$ denotes the standard unit sphere.
\end{thm}

However, the above constant $\delta(m)$ is implicitly obtained from a contradiction argument. So it is desirable to determine the sharp value of $\delta(m)$. To the best of the authors' knowledge, there has not been much progress in this direction. When $m\leq 3$ one can argue that $\delta(m)=1/2$ by using classification results in low-dimensional topology. But when $m\geq 4$ even a rough lower bound for $\delta(m)$ is still missing. When $m=4$ we expect that the complex projective plane $\mathbb P^2$ attains the second largest volume, but a proof for this seems to be out of reach. We refer the reader to \cite{HM26} for more detailed discussions on this topic.

In K\"ahler geometry, we also have an almost rigidity result, due to Y. Liu \cite[Theorem A.1]{Zhang22}.

\begin{thm}\label{thm:Liu}
    There exists a dimensional constant $\varepsilon(n)>0$ such that the following holds. Let $(X^n,\omega)$ be a compact K\"ahler manifold with $\Ric(\omega)\geq (n+1) \omega$. Assume that $$\vol(X,\omega)>(1-\varepsilon(n))\vol(\IP^n,\omega_{\fs}),$$ then $X$ is biholomorphic to $\IP^n$. Here  $\omega_{\fs}$ denotes the standard Fubini--Study metric on the complex projective space $\mathbb P^n$, normalized so that $\vol(\IP^n,\omega_{\fs})=\int_{\mathbb P^n}\omega^n_{\fs}=(2\pi)^n$.
\end{thm}

As in the Riemannian case, the constant $\varepsilon(n)$ was only implicitly obtained.
The goal of this work is to determine the sharp value of $\varepsilon(n)$.
Note that when $n=1$, the Ricci curvature condition alone readily implies that $X\cong\IP^1$. So we only need to consider the case when $n\geq 2$.
Our main result is the following.

\begin{thm}\label{thm:gap}
    The optimal constant $\varepsilon(n)$ in Theorem \ref{thm:Liu} is $\varepsilon(n)=1-\frac{2n^{n}}{(n+1)^n}.$ Moreover, if $(X^n,\omega)$ is a compact K\"ahler manifold with $\Ric(\omega)\geq (n+1) \omega$ and 
    \begin{equation}
        \label{eq:vol=2nn}
        \vol(X,\omega)=\frac{2n^{n}}{(n+1)^n}\vol(\IP^n,\omega_{FS}),
    \end{equation} then $(X,\omega)$ is  biholomorphically isometric to the quadric hypersurface $Q^n\subset\mathbb P^{n+1}$ or $\mathbb P^1\times\mathbb P^{n-1}$, equipped with their standard K\"ahler--Einstein metrics with Einstein constant $(n+1)$.
\end{thm}

When $n=2$, the Ricci curvature condition implies that $X$ is del Pezzo and that $\vol(X,\omega)\leq \frac{(2\pi)^2}{9}(-K_X)^2$. Then, due to the standard classification of del Pezzo surfaces, our result follows immediately (note also that in this case $Q^2\cong \IP^1\times \IP^1$). Therefore, throughout this paper, we always assume that $n\geq 3$.

To explain the idea of the proof, first note that the positive Ricci curvature condition ensures that $X$ is a Fano manifold, whose K\"ahler cone coincides with the ample cone, so $\frac{1}{2\pi}[\omega]$ can be viewed as the numerical class of an ample $\IR$-line bundle $L$.
We denote by $\beta(X,L)$ the greatest Ricci lower bound of $L$, defined as
$$
\beta(X,L):=\sup\{t>0: \text{there exists }\omega\in2\pi c_1(L)\text{ with }\Ric(\omega)\geq t\omega\}.
$$
Then the positive Ricci curvature condition implies that $\beta(X,L)\geq n+1$.
In order to prove Theorem \ref{thm:gap} we will mainly investigate the \emph{$\beta$-volume}:
$
\beta(X,L)^n\cdot \vol(L).
$
A key property of the $\beta$-volume is its invariance under rescaling, in the sense that $\beta(X,rL)^n\cdot\vol(rL)=\beta(X,L)^n\cdot \vol(L)$ for $r\in \IR_{>0}$. Moreover, in \cite{Zhang22} it was shown that $\beta(X,L)^n\cdot \vol(L)\leq (n+1)^n$, with equality holding if and only if $X\cong\mathbb P^n$. For more details on this quantity, we refer to \S \ref{sec:pre}. In this work we establish the following gap theorem for the $\beta$-volume, which will imply Theorem \ref{thm:gap}. 
\begin{thm}\label{thm:rational}
   Let $X$ be an $n$-dimensional smooth Fano manifold and $L$ be an ample $\IR$-line bundle. If $X$ is not biholomorphic to $\IP^n$, then $\beta(X,L)^n\cdot \vol(L)\leq 2n^n$. The equality occurs only when $X\cong Q^n$ or $X\cong \IP^1\times \IP^{n-1}$. 
\end{thm}
Theorem $\ref{thm:rational}$ extends the main result of \cite{LM25} where $L=-K_X$ and $X$ is assumed to be K-semistable (in which case $\beta(X,-K_X)=1$ by \cite[Theorem 1.5]{Li17}). The proof of Theorem \ref{thm:rational} follows the same route as in \cite{LM25}, building on the volume estimate in \cite{Fvol} and the existence of minimal rational curves $C$ on Fano manifolds due to Mori \cite{Mor79}.

However, a crucial difficulty arises for the characterization of the equality case $\beta(X,L)^n\cdot L^n= 2n^n$, since it is possible that this equality is attained by some irrational class,  making the analysis for $\mathbb Q$-line bundles in \cite{LM25} inapplicable. To deal with this subtlety, we will follow the strategy in \cite[\S 4]{Zhang22}, by using the positivity theory of Okounkov bodies \cite{KL17}.

In fact, the same method allows us to prove a more general volume estimate for twisted K\"{a}hler-Einstein metrics.  
\begin{thm}\label{thm-twFano}
     Assume $(X, \eta)$ is a twisted Fano pair (see Definition \ref{defi:twisted}) with $X\not\cong \IP^n$ and $L$ be an ample $\IR$-line bundle. Then either $\beta_\eta(X, L)^n\cdot \vol(L)\le 2n^n$ (see \eqref{eq-beeta}), or every minimal rational curve $C$ on $X$ has anti-canonical degree $-K_X\cdot C\ge  n$.

    If furthermore $X$ is itself a Fano manifold not isomorphic to $\IP^n$ and $\eta=2\pi\{\Delta\}+\eta_1$ where $\Delta$ is an effective $\mathbb Q$-divisor and $\eta_1\ge 0$ is a closed positive (1,1)-current with a nef class $[\eta_1]$, then $\beta_\eta(X, L)^n\cdot \vol(L)\le 2n^n$, and the equality occurs only when $X\cong Q^n$ or $\IP^1\times \IP^{n-1}$ and $\eta=0$.
\end{thm}

In the toric case, similar to \cite[Section 3]{LM25}, we can use the fact that any toric manifold has an embedded projective space (of lower dimension) with trivial normal bundles (\cite{CFH14}) and directly prove the sharp volume gap estimates for K-semistable toric log Fano manifolds (see Theorem \ref{thm:toricPair}). Combined with a logarithmic generalization of Moraga-S\"{u}ss' gap estimates \cite{MS24} (see Appendix \ref{app-vol}), we can prove the gap hypothesis for all K-semistable toric log Fano pairs. This conjecture was first proposed in \cite[(1.3)]{AB25} and was verified under some conditional assumptions in \cite[Theorem 3.1]{AB25}. We refer our readers to a series of work \cite{AB24,Ber25}) for the arithmetic application. 
\begin{thm}\label{thm:toricPair}
    Let $(X,\D)$ be any K-semistable toric log Fano pair with $X\ncong \IP^n$. We have the sharp estimate:
    $$
    (-K_X-\D)^n\leq (-K_{\IP^1\times \IP^{n-1}})^n=2n^n.
    $$ 
    Moreover, the equality holds only when $(X, \Delta)\cong (\IP^1\times \IP^{n-1}, 0)$. 
\end{thm}

\subsection*{Acknowledgments}
C. Li is partially supported by NSF (Grant No. DMS-2305296).
M.Miao is partially supported by NSFC grant 125B2003. K.Zhang is partially supported by Scientific Research Innovation Capability Support Project for Young Faculty SRICSPYF-ZY2025169, and also by NSFC grants 12571060, 12271038 and 12271040.
M.Miao would like to express his sincere gratitude to his advisor, Gang Tian, for his constant encouragement and support. This project was initiated when two of the authors were attending Zhuhai's Workshop on Geometric Analysis in May 2026. They would like to thank the organizers for their warm hospitality and for providing an inspiring environment. We thank R. Andreasson, R. Berman, X. Wang, J. Zhou and J. Zhu for many helpful discussions. 

\subsection*{Declaration on the Use of AI}
All mathematical ideas, arguments, and results presented in this manuscript
were developed by the authors. After completion of the initial draft, AI
tools were used only to improve the language and presentation and to assist in
identifying possible typographical, expository, and logical issues. The authors take
full responsibility for the content of the manuscript.

\section{Preliminaries}
\label{sec:pre}

\subsection{The $\beta$-volume of twisted Fano pairs}
Let $X$ be a smooth projective manifold.  
Let $\eta$ be a closed positive $(1, 1)$-current. It is well known that locally $\eta=\sddb \psi$ for a plurisubharmonic function $\psi$. To incorporate both cases of smooth K\"{a}hler metrics with positive Ricci curvature and conical K\"{a}hler-Einstein metrics, we will work in the general twisted setting as in \cite{BBJ}. 
\begin{defi}\label{defi:twisted}
    We say that a closed positive (1,1)-current $\eta$ is a klt current if 
    $e^{-\psi}$ is locally integrable, and if furthermore $2\pi c_1(X)-[\eta]\in H^{1,1}(X, \IR)$ is a positive class, we say that $(X, \eta)$ is a twisted Fano pair. 
\end{defi}
\begin{ex}\label{ex-sm}
If $\eta\ge 0$ is smooth, $(X, \eta)$ is a twisted Fano pair if and only if $2\pi c_1(X)-[\eta]$ is a positive class. In this case, $c_1(X)$ is also a positive class and hence $X$ is a Fano manifold. When $\eta=0$, we will simply write $X$ for $(X, 0)$. 
\end{ex}
\begin{ex}\label{ex-con}
Let $\Delta=\sum_i a_i \Delta_i$ be a $\IQ$-divisor with $a_i\in (0,1)$. Set $\eta_\Delta=2\pi\sum_i a_i\{\Delta_i\}$ where $\{\Delta_i\}$ is the distribution of integration along $\Delta_i$. 
    Then $\eta_\Delta$ is a klt current if and only if $(X, \Delta)$ has klt singularities, if and only if
    $\prod_i |f_i|^{-2a_i}$ is locally integrable where $f_i$ is a local defining function of $\Delta_i$, i.e. $\Delta_i=\{f_i=0\}$. 
    We say that $(X, \Delta)$ is a log Fano pair if $(X, \Delta)$ has klt singularities and $-(K_X+\Delta)$ is an ample $\IQ$-Cartier divisor. In other words, $(X, \Delta)$ is a log Fano pair if and only if $(X, \eta_\Delta)$ is a twisted Fano pair. 
\end{ex}

If $(X, \eta)$ is a twisted Fano pair, then $-K_X=-(K_X+\frac{1}{2\pi}[\eta])+\frac{1}{2\pi}[\eta]$ admits a K\"{a}hler current with trivial multiplier ideal. 
By the Nadel vanishing theorem, we have:
\begin{equation*}
    H^i(X, \CO_X)=0, \quad \text{for } i>0. 
\end{equation*}

This implies that the K\"ahler cone of $X$ coincides with its ample cone. Therefore, any K\"ahler class $\frac{1}{2\pi}[\omega]\in H^{1,1}(X,\mathbb R)$ can be identified with the numerical class of an ample $\mathbb R$-line bundle $L$ in the Néron--Severi space $N^1(X)_{\mathbb R}$. 
Recall that an $\IR$-line bundle (resp. $\IQ$-line bundle) is an element of $\text{Pic}(X)\otimes_{\IZ}\IR$ (resp. $\text{Pic}(X)\otimes_{\IZ}\IQ$) where $\text{Pic}(X)$ is the Picard group of $X$.

In this setting, the volume of a K\"{a}hler manifold $(X,\omega)$, with $\omega\in 2\pi c_1(L)$, is simply given by
$$
\vol(X,\omega)=(2\pi)^n\cdot L^n=(2\pi)^n\cdot \vol(L).
$$
Here $\vol(L)$ denotes the algebro-geometric volume of $L$. When $L$ is a $\mathbb Q$-line bundle, it is defined as
\begin{eqnarray}\label{def:volume}
\vol(L):=\mathop{\limsup}_{m\to\infty}\frac{h^0(X,mL)}{m^n/n!}, 
\end{eqnarray}
which can be further extended to $\mathbb R$-line bundles by continuity. For more details, we refer to  \cite[Section 2.2.C]{Lbook}.

Fix an ample $\IR$-line bundle $L$ and choose a smooth K\"{a}hler form $\omega_0\in 2\pi c_1(L)$.
Denote by $\mathrm{PSH}_{\infty}(\omega_0)$ the set of bounded $\omega_0$-plurisubharmonic functions on $X$. For $\omega_\vphi=\omega_0+\sddb \vphi$ with $\vphi\in \mathrm{PSH}_\infty(\omega_0)$, its Monge-Amp\`{e}re measure $\omega_\vphi^n$ is well defined. Fix a smooth volume form $\Omega$ on $X$. If 
there exists $f\in L^1(X, \Omega)$ such that $\omega_\vphi^n=e^{-f}\Omega$, then we say that $\Ric(\omega_{\varphi}):= -\sqrt{-1}\partial\bar\partial \,\log(\omega_{\varphi}^n)=\sqrt{-1}\partial\bar\partial f+\ric(\Omega)$ is well defined. In this case, for any klt current $\eta$, define $\Ric_\eta(\omega_\vphi)=\Ric(\omega_\vphi)-\eta$. 
Consider the following \emph{greatest twisted Ricci lower bound}:
\begin{equation}\label{eq-beeta}
    \beta_\eta(X, L)=\sup\{t>0; \exists\, \omega\in 2\pi c_1(L) \,\text{such that} \,\Ric_\eta(\omega)\ge t \omega  \}. 
\end{equation}
Here $\Ric_\eta(\omega)\ge t\omega$ means that $\Ric_\eta(\omega)$ is well defined for $\omega=\omega_\vphi$ for some $\vphi\in \mathrm{PSH}_\infty(\omega_0)$ and $\Ric_\eta(\omega)-t\omega$ is a smooth semi-positive (1,1)-form. 
When $\eta=0$, we simply denote it as $\beta(X,L)$.

This invariant originates from Tian \cite{T92}. It only depends on the numerical class $c_1(L)$ and has been studied intensively in the literature; see e.g. \cite{R08,Sz11,Li11,SW16} in the case where $\eta=0$ and $L=-K_X$.
Thanks to the recent developments in the field of Yau--Tian--Donaldson conjecture,
we can now characterize the $\beta_\eta$ quantity in terms
of algebraic invariants. 

To explain this characterization, we first define the nef threshold of $(X, \eta)$ with respect to $L$ as:
\begin{equation*}
    s_\eta(X, L)=\sup\{t>0; c_1(X, \eta)-tc_1(L) \text{ is a nef class} \}, 
\end{equation*}
where $c_1(X,\eta):=c_1(X)-\frac{1}{2\pi}[\eta]$.

For any prime divisor $E\subset Y\xrightarrow{\pi}X$ on a smooth birational model $Y$ over $X$, define its twisted log discrepancy as
$$
A_\eta(E):=1+\ord_E(K_Y-\pi^*K_X)-\frac{1}{2\pi}\nu(\eta,E)=A_X(E)-\frac{1}{2\pi}\nu(\eta,E)
$$
where $\nu(\eta,E)$ denotes the Lelong number of the current $\pi^*\eta$ along $E$. For instance, when $\eta$ is from Example \ref{ex-con}, one simply has $\frac{1}{2\pi}\nu(\eta_\Delta,E)=\ord_E(\Delta)$; when $\eta$ is from Example \ref{ex-sm}, one has $\nu(\eta,E)=0$, so that $A_\eta(E)=A_X(E)$ is the usual log discrepancy in birational geometry. In any case, we always have that
$$
A_\eta(E)\leq A_X(E).
$$
Define the twisted $\delta$-invariant:
\begin{equation*}
    \delta_\eta(X, L)=\mathop{\inf}_E \frac{A_\eta(E)}{S(L;E)}, 
\end{equation*}
where $E$ runs through all prime divisors on all models over $X$. 
The definition when $L$ is $\IQ$-line bundle and $\eta=0$ is due to \cite{FO18,BJ20}; see \cite{Zhang21} for the extension to $\IR$-line bundles. Here $E\subset Y\xrightarrow{\pi}X$ runs through all prime divisors on smooth birational models $Y$ over $X$ and 

$$
S(L;E):=\frac{1}{\vol(L)}\int_0^\infty\vol(\pi^*L-xE)\, \dif x
$$
is the expected vanishing order of $L$ along $E$.

\begin{thm}[\cite{BBJ,CRZ,Zhang21}]
\label{thm:bbj}

For any ample $\mathbb R$-line bundle $L$
\begin{equation}
\label{eq:beta=min-s-delta}
    \beta_\eta(X, L)=\min\{s_\eta(X, L), \delta_\eta(X, L)\}. 
\end{equation}
\end{thm}
This statement follows from the resolution of uniform Yau--Tian--Donaldson conjecture for twisted K\"{a}hler-Einstein metrics. 
We say that $(X, \eta)$ admits a twisted K\"{a}hler-Einstein metric if $c_1(X,\eta)=c_1(X)-\frac{1}{2\pi}[\eta]$ is a K\"{a}hler class and there exists a K\"ahler metric $\omega\in 2\pi c_1(X,\eta)$ such that
\begin{equation}
    \Ric_\eta(\omega):=\Ric(\omega)-\eta=\omega. 
\end{equation}

More generally, given an ample $\mathbb R$-line bundle $L$ and any $s<s_\eta(X, L)$, one can choose a smooth semi-positive $(1,1)$-form $\theta\in 2\pi (c_1(X, \eta)-s c_1(L))$. Consider the following equation of twisted K\"{a}hler-Einstein metrics:
\begin{equation}\label{eq-KEs}
    \Ric_\eta(\omega)=s \omega+\theta,\ \omega\in 2\pi c_1(L). 
\end{equation}
Thanks to uniform Yau-Tian-Donaldson theorem, we can characterize the solvability of \eqref{eq-KEs} in terms of algebraic thresholds. Conversely, suppose that we have some $\omega\in 2\pi c_1(L)$ satisfying $\Ric_\eta(\omega)\geq s\omega$, then by our definition $\theta:=\Ric_\eta(\omega)-s\omega$ is a smooth semi-positive $(1,1)$-form and $\omega$ solves \eqref{eq-KEs}.

To prove the main results of this paper, we will only need one part (which is the easier direction) of Theorem \ref{thm:bbj}:

\begin{prop}
    Suppose that there exists $\omega\in 2\pi c_1(L)$ such that $\Ric_\eta(\omega)\geq s\omega$, then $c_1(X, \eta)-s c_1(L)$ is nef and $\delta_\eta(X, L)\geq s$. Namely, $\beta_{\eta}(X,L)\leq \min \{s_{\eta}(X,L), \delta_{\eta}(X,L)\}$.
\end{prop}

Now for any ample $\mathbb R$-line bundle $L$ over a twisted Fano pair, we define its \emph{$\beta$-volume} to be $$\beta_\eta(X, L)^n\cdot\vol(L).$$ It is straightforward to verify that this quantity is scaling invariant. To prove Theorem \ref{thm:gap}, one essentially needs to get an upper bound for the $\beta$-volume. In fact, when $\eta=0$, the following estimate  was previously shown in \cite{Zhang22}, building on ideas from \cite{Fvol}.

\begin{thm}[\cite{Zhang22}]
    For any ample $\mathbb R$-line bundle $L$, one has $\beta(X,L)^n\cdot\vol(L)\leq (n+1)^n$, and the equality holds only when $X\cong\mathbb P^n$.
\end{thm}

In \cite[Theorem A.1]{Zhang22}, Y. Liu further showed the existence of a dimensional constant $\varepsilon(n)>0$ so that $\beta(X,L)^n\cdot\vol(L)\leq (1-\varepsilon(n))(n+1)^n$ whenever $X\not\cong\mathbb P^n$. One of the main goals of this paper is to determine the sharp value of $\varepsilon(n)$, and also extend the discussion to the more general twisted setting with $\eta\ge 0$.
To this end, we will rely on two important tools in algebraic geometry, \emph{minimal rational curves} and \emph{Okounkov bodies}, which we now turn to describe.

\subsection{Minimal rational curves and weighted blowups}\label{sec-wtbl}
A smooth projective manifold $X$ is called \textit{uniruled} if there exists a dominant rational map: $Y\times \IP^1\dashrightarrow X$ where $Y$ is an $(n-1)$-dimensional projective variety. A fundamental result in birational algebraic geometry states that $X$ is uniruled if and only if $K_X$ is not pseudoeffective (\cite{MM86,BDPP13}). Basic examples of uniruled manifolds are Fano manifolds. More generally, if $(X, \eta)$ is a twisted Fano pair, then $X$ is uniruled. Indeed, in this case $-K_X=c_1(X, \eta)+\frac{1}{2\pi}[\eta]$ is big and hence $K_X$ cannot be pseudoeffective (otherwise the zero divisor would be big). 

The uniruled property can be characterized by the existence of free rational curves. 
Recall that a \textit{rational curve} on $X$ is a non-constant holomorphic morphism $f: \IP^1\rightarrow X$. A rational curve $f$ is called \textit{free} if 
\begin{equation*}
    f^*T_X=\oplus_{i=1}^n \CO_{\IP^1}(a_i)
\end{equation*}
with $a_i\ge 0$, $i=1,\dots, n$. It is well known that $X$ is uniruled if and only if there exists a free rational curve on $X$. 
By applying Mori's technique of bend-and-break to free rational curves, we can obtain \textit{minimal rational curves} which are represented by morphisms $f: \IP^1\rightarrow X$ such that
\begin{equation}\label{eq-mrat}
    f^*T_X=\CO_{\IP^1}(2)\oplus \CO_{\IP^1}(1)^{\oplus (d-2)}\oplus \CO_{\IP^1}^{\oplus (n+1-d)},
\end{equation}
where $d=-K_X\cdot f(\IP^1)$ (See \cite[Theorem IV.2.10]{K96}). In fact, a uniruled manifold is always covered by minimal rational curves. 

Any minimal rational curve satisfying \eqref{eq-mrat} is always an immersed curve and its degree satisfies $$d=\deg(f)=-K_X\cdot f(\IP^1)=-K_X\cdot C\in \{2,3,\cdots,n+1\}.$$ Its normal bundle is defined as 
\begin{equation}\label{eq-nsplit}
N_{f/X}=f^*T_X/T\IP^1=\CO_{\IP^1}(1)^{\oplus (d-2)}\oplus \CO_{\IP^1}^{\oplus (n+1-d)}.
\end{equation}

By the definition of nef threshold $s_\eta(X, L)$, we have $(c_1(X)-\frac{1}{2\pi}[\eta]-s_\eta(X, L)\cdot L)\cdot f(\IP^1)\geq 0$. So if we denote $$p\coloneqq L\cdot f(\IP^1),$$ then 
\begin{eqnarray}\label{eq:dsp}
    s_\eta(X, L)\cdot p\le (-K_X-\frac{1}{2\pi}[\eta])\cdot C\le -K_X\cdot C=d.
\end{eqnarray}
The second inequality holds because the curve $C$ is free and hence movable.

In \cite{LM25}, for any minimal rational curve of degree $d$, two divisorial valuations $v_\ell$, $\ell=1, 2$, are canonically defined as a $(1^{\oplus(d-2)},\ell^{\oplus (n+1-d)})$-weighted blowup that is adapted to the splitting \eqref{eq-nsplit}. More precisely, choose any point $t\in \IP^1$ and a smooth local branch $C_1$ passing through $f(t)=x$ (if $C$ is smooth at $f(t)$ then $C_1=C$ locally), we can choose local coordinates $\mathbf{z}=\{z_1,\dots, z_n\}=(\mathbf{z}', \mathbf{z}'', z_n)$ with $\mathbf{z}'=\{z_1,\dots, z_{d-2}\}$, $\mathbf{z}''=\{z_{d-1},\dots, z_{n-1}\}$ that satisfy:
\begin{enumerate}
    \item $z_i(f(t))=0$ for $i=1,\dots, n$;
    \item Locally $C_1$ is defined by $\{z_1=\cdots=z_{n-1}=0\}$;
    \item For any $1\le i\le d-2$, $f^*\dif z_i$ is a local generator of the $i$-th summand of $\CO(-1)^{\oplus (d-2)}\hookrightarrow N^\vee_{f/X} $;
    \item For any $d-1\le j\le n-1$, $f^*\dif z_j$ is the $j$-th summand of $\CO^{\oplus (n-d+1)}\hookrightarrow N^\vee_{f/X}$. 
\end{enumerate}
We will call such coordinates adapted to the splitting \eqref{eq-nsplit}.
For any holomorphic function $h$ on $X$ that is holomorphic in a neighborhood of $x\in X$, we expand
$h=h(z)$ locally as:
\begin{equation*}
    h(z)=\sum_{I,J}b_{IJ}(z_n)\mathbf{z}'^I\mathbf{z}''^J
\end{equation*}
where $\mathbf{z}'^I=z_1^{i_1}\cdots z_{d-2}^{i_{d-2}}$
(resp. $\mathbf{z}''^J=z_{d-1}^{i_{d-1}}\cdots z_{n-1}^{i_{n-1}}$), 
for $I=\{i_1,\dots, i_{d-2}\}\in \IN_{\ge 0}^{d-2}$ with $|I|=i_1+\cdots+i_{d-2}$ (resp. $J=\{j_{d-1},\dots, j_{n-1}\}\in \IN_{\ge 0}^{n-d+1}$ with $|J|=j_{d-1}+\cdots j_{n-1}$). 
For each $\ell\in \{1,2\}$, define:
\begin{equation}
v_\ell(h)=\min\left\{|I|+\ell|J|; b_{IJ}(z_n)\not\equiv 0\right\}.
\end{equation}
One can show that $v_\ell$ is well-defined and does not depend on the choice of $t\in \IP^1$, the local branch $C_1$ passing through $f(t)$ and local adapted coordinates (see \cite[Section 4]{LM25} for a complete treatment). 
Moreover $v_\ell=\ord_E$ is a divisorial valuation where $E=E_\ell$ is a prime divisor on a birational model $\pi\colon \hat{X}\rightarrow X$ of $X$. 
The center of $v$ over $X$ (image of $E$ on $X$) is $C=f(\IP^1)$. As in \cite{LM25}, we will use $v_1$ for $d=2$ and $v_2$ for $d\ge 3$ in the following estimates. 
\begin{prop}\label{prop:phi}
    With the above notation, for any $\IR$-line bundle with $s_\eta(X, L)\ge 1$, we have:
    \begin{equation}\label{eq-volest}
    \vol(\pi^*L-x E)\geq\vol(L)-\phi_d(x),\ x\geq 0
    \end{equation}
    where $\phi_d(x)$ has the following expression in each of the three cases:
    \begin{enumerate}
    \item  When $d=2$ and $C$ is an embedded minimal rational curve:
    $\phi_d(x)=n d x^{n-1}=2 n x^{n-1}$.
    \item
    When $d=2$ and $C$ is an immersed minimal rational curve with a singular image,
         \begin{eqnarray*}
    \phi_d(x)=
\begin{cases}
    ndx^{n-1}-x^n & \text{if} \,\, x\leq d,\\
   ndx^{n-1}-x^n+(x-d)^n  & \text{if} \,\, x> d. 
\end{cases}
\end{eqnarray*}
\item When $3\leq d\leq n+1$ and $C$ is a minimal rational curve, 
\begin{eqnarray*}
    \phi_d(x)=
    \begin{cases}
        \frac{1}{2^{n-d+1}}\left(ndx^{n-1}-(d-2)x^n \right) & \text{if}\,\,\, x\leq d, \\
        \frac{n!}{2^{n-d+1} (n-d+1)!(d-3)!}\int_0^d \beta^{d-3}(d-\beta)(x-\beta)^{n-d+1}\, \dif \beta & \text{if} \,\,\, x>d.
    \end{cases}
\end{eqnarray*}
    \end{enumerate}
\end{prop}

\begin{proof}
We only outline the proof and refer to \cite{LM25} for detailed calculations. We first treat the case when $L$ is an ample $\mathbb Q$-line bundle, and then extend to $\mathbb R$-line bundles by continuity.

For any $y\in \IR_{\ge 0}$, we define the valuative ideal sheaf $\CI_{y}=\{f\in \CO_X; v_\ell(f)\ge y\}\subset \CO_X$. 
Given an ample $\mathbb Q$-line bundle $L$ and $x\in \IQ_{\geq 0}$, we choose $k\in \IN^*$ sufficiently large such that $xk\in \IZ_{\geq 0}$ and $kL$ is a line bundle, then we have the following long exact sequence:
\begin{eqnarray*}
    0\rightarrow H^0(X, kL\otimes \CI_{xk})\rightarrow H^0(X, kL)\rightarrow H^0({X}, kL\otimes \CO_X/\CI_{xk})\rightarrow \cdots,
\end{eqnarray*}
which implies
$$\vol(\pi^*L-xE)\ge L^n-\mathop{\limsup}_{k\rightarrow+\infty} \frac{h^0({X}, kL\otimes \CO_X/\CI_{xk})}{k^n/n!}.$$
From the exact sequence:
\begin{equation*}
    0\rightarrow \CI_y/\CI_{y+1}\rightarrow \CO_X/\CI_{y+1}\rightarrow \CO_X/\CI_y\rightarrow 0,
\end{equation*}
we inductively deduce the estimates:
\begin{eqnarray}\label{eq-tels2}
h^0(X, kL\otimes \CO_X/\CI_{xk})
&\le&\sum_{y=0}^{xk-1}h^0(X, kL\otimes \CI_{y}/\CI_{y+1}).
\end{eqnarray}
Thus,
\begin{eqnarray*}
    \vol(\pi^*L-xE)
    &\ge&L^n-\mathop{\limsup}_{k\rightarrow+\infty}
\sum_{y=0}^{xk-1}\frac{h^0(X, kL\otimes \CI_y/\CI_{y+1})}{k^n/n!}. 
\end{eqnarray*}
To estimate the second term on the right-hand side, as in \cite{LM25}, we consider three cases.

\textbf{(Case I):} When $d=2$ and $C$ is an embedded minimal rational curve, we consider the valuation $v=v_1$, which is the divisorial valuation induced by the exceptional divisor $E$ of the ordinary blowup $\pi\colon \bl_C X \rightarrow X$ along the rational curve $C$ with trivial normal bundle (see \eqref{eq-nsplit}). Denote by $\CI=\CI_C$ the ideal sheaf of $C$, then $\CI_{y}=\CI^y$ and we get the following estimate:
    \begin{eqnarray*}
        \vol(\pi^*L-xE) \geq L^n -\mathop{\limsup}_{k\rightarrow +\infty} \frac{n!}{k^n}\sum_{y=0}^{xk-1}h^0(\IP^1, kL|_C\otimes \CI^y/\CI^{y+1}).
    \end{eqnarray*}
    Moreover, $\CI^y/\CI^{y+1}\cong\sym^y(N_{C/X}^{\vee})=\sym^y(\CO_{C}^{\oplus(n-1)})\cong \CO_C^{\oplus \binom{n-2+y}{n-2}}$. So we get:
    \begin{eqnarray*}
       &&\mathop{\limsup}_{k\rightarrow +\infty} \frac{n!}{k^n}\sum_{y=0}^{xk-1}h^0(\IP^1, kL|_C\otimes \CI^y/\CI^{y+1})=
        \mathop{\limsup}_{k\rightarrow +\infty} \frac{n!}{k^n}\sum_{y=0}^{xk-1}\binom{n-2+y}{n-2}(kp+1)\\
        &=& n(n-1)\cdot\mathop{\limsup}_{k\rightarrow +\infty}\sum_{y=0}^{xk-1} \left(\frac{y+1}{k}\right)\cdots \left(\frac{y+n-2}{k}\right)\left(p+\frac{1}{k}\right) \frac{1}{k}\\
        &=& n(n-1)p\int_0^x \alpha^{n-2}\, \dif \alpha=npx^{n-1}\le n d x^{n-1}=:\phi_d(x).
    \end{eqnarray*}
    For the last inequality we used the assumption $s_\eta(X, L)\ge 1$ so that $p\le d$ from (\ref{eq:dsp}).

\textbf{(Case II):}  When $d=2$ and $C$ is an immersed minimal rational curve with a non-empty subset of singular points, then the normalization map $f: \IP^1\rightarrow f(\IP^1)$ satisfies $f(t)=f(t')$ for some $t\neq t'\in \IP^1$. Any section $[s]\in H^0(X, kL\otimes \CI_y/\CI_{y+1})$ still comes from a section $\tilde{s}$ in $H^0\left(\IP^1, \CO(pk)\otimes \mathrm{Sym}^y(\CO_C^{\oplus (n-1)})\right)$. 
    However, this section $\tilde{s}$ must satisfy non-trivial relation at $t$ and $t'$ which leads to the strong inclusion $\tilde{s}\in H^0(\IP^1, \CO(pk)\otimes \mathrm{Sym}^y(\CO(-1)\oplus \CO^{\oplus (n-2)}))$ (see more details in \cite[Section 4.2]{LM25}). With the assumption $s_\eta(X,L)\ge 1$, we know that $p\le d$ and hence 
    $$
    h^0(\IP^1, \CO(pk)\otimes \mathrm{Sym}^y(\CO(-1)\oplus \CO^{\oplus (n-2)}))\le 
    h^0(\IP^1, \CO(d k)\otimes \mathrm{Sym}^y(\CO(-1)\oplus \CO^{\oplus (n-2)})).
    $$ 
    So we get a stronger estimate:
    \begin{eqnarray*}
       \sum_{y=0}^{xk-1}h^0(X, kL\otimes \CI_y/\CI_{y+1})&\le&       
        \sum_{y=0}^{xk-1}\sum_{i=0}^{\min\{y,dk\}}(dk-i+1)\binom{n-3+y-i}{n-3}\\
       &=&\phi_d(x) \frac{k^n}{n!}+O(k^{n-1})    
    \end{eqnarray*}
where, by the same calculation as in \cite[Section 4.2]{LM25}, we have:
    \begin{eqnarray*}
    \phi_d(x)=
\begin{cases}
    2nx^{n-1}-x^n & \text{if} \,\, x\leq 2\\
   2nx^{n-1}-x^n+(x-2)^n  & \text{if} \,\, x> 2. 
\end{cases}
\end{eqnarray*}

\textbf{(Case III):} When $3\leq d\leq n+1$ and $C$ is a minimal rational curve. Assuming $s_\eta(X, L)\ge 1$, we again obtain $p=L\cdot f(\IP^1)\le d$. 
   By the method in \cite[Section 4.3]{LM25}, we get the estimate:
    \begin{eqnarray*}
       && \sum_{y=0}^{xk-1}h^0(X, kL\otimes \CI_y/\CI_{y+1})\\
       &\le& \sum_{y=0}^{xk-1}\sum_{m=0}^{\lfloor y/2\rfloor}
        h^0(\IP^1, \CO(kp)\otimes \mathrm{Sym}^{y+2m-2\lfloor y/2\rfloor}(\CO(-1)^{\oplus (d-2)})\otimes \mathrm{Sym}^{\lfloor y/2\rfloor-m}(\CO^{\oplus (n-d+1)}))\\
    &\le & \sum_{y=0}^{xk-1}\sum_{m=0}^{\lfloor\min\{y,kd\}/2\rfloor}\binom{d-3+2m+t_y}{d-3}\binom{n-d+\lfloor y/2 \rfloor -m}{n-d} h^0(\IP^1, \CO(kd-2m))
    \end{eqnarray*}
    where $t_y=1$ if $y$ is odd and $t_y=0$ if $y$ is even, and for the last inequality we used the inequality $h^0(\IP^1, \CO(kp-2m-t_y)\le h^0(\IP^1, \CO(kd-2m))$. 
    By the same calculation as in \cite{LM25}, we get the asymptotic expression of last expression: 
    $\phi_d(x)\frac{k^n}{n!}+O(k^{n-1})$ where
    \begin{eqnarray*}
    \phi_d(x)=
    \begin{cases}
        \frac{1}{2^{n-d+1}}\left(ndx^{n-1}-(d-2)x^n \right) & \text{if}\,\,\, x\leq d, \\
        \frac{n!}{2^{n-d+1}(n-d+1)!(d-3)!}\int_0^d \beta^{d-3}(d-\beta)(x-\beta)^{n-d+1}\, \dif \beta & \text{if} \,\,\, x>d.
    \end{cases}
\end{eqnarray*}
Hence, the desired estimate \eqref{eq-volest} holds when $L$ is a $\IQ$-line bundle. 

To see that the same estimate holds for an $\mathbb R$-line bundle $L$ satisfying $s_\eta(X,L)\geq 1$, one can approximate $L$ using a sequence of ample $\mathbb Q$-line bundles $L_i$ with $L-L_i$ nef. Then one also has $s_\eta(X,L_i)\geq 1$, and we conclude by the continuity of the volume function in $N^1(\hat X)_{\mathbb R}$ \cite[Section 2.2.C]{Lbook}.
\end{proof}

\subsection{Okounkov bodies}

As in \cite{Zhang22}, to better understand the volume function $\vol(\pi^*L-xE)$ appearing above, we will use the positivity theory of Okounkov bodies developed in \cite{KL17}.

Let $Y$ be a smooth projective manifold of dimension $n$. Let $L$ be a big line bundle on $Y$ (i.e., $\vol(L)>0$). An admissible flag
$$Y_\bullet: Y=Y_0\supset Y_1\supset...\supset Y_n=\{y\}$$
is a complete flag of irreducible subvarieties of $Y$ where each $Y_i$ is of codimension $i$ and is smooth at the point $Y_n$. Following \cite{LM09}, any flag $Y_\bullet$ induces a valuative map:
$$
\nu_{Y_\bullet}: H^0(Y,mL)\setminus \{0\}\to\mathbb N^n.
$$
More precisely, for any nonzero $s\in H^0(Y,mL)$, let $k_1:=\ord_{Y_1}(s)$. Dividing the $k_1$-th power of the defining section of $Y_1$, we get a section $s'\in H^0(Y,mL-k_1 Y_1)$ which does not vanish identically on $Y_1$. So by restriction we get $s_1\in H^0(Y_1,(mL-k_1 Y_1)|_{Y_1})$. Put $k_2:=\ord_{Y_2}(s_1)$. Continuing in this way, we obtain the map $\nu_{Y_\bullet}(s)=(k_1,...,k_n)\in\mathbb N^n$.

The Okounkov body of $L$ is then defined to be
$$
\Delta(L):=\Delta_{Y_\bullet}(L):=\mathrm{closure}\left(\bigcup_{m\geq 1}\frac{1}{m}\left\{\nu_{Y_\bullet}(s):s\in H^0(Y,mL)\setminus \{0\}\right\}\right).
$$
A crucial fact is that $\Delta(L)$ is a convex body in $\mathbb R^n$ and
$$
\vol_{\mathbb R^n}(\Delta(L))=\frac{\vol(L)}{n!}.
$$
Note that $\Delta(L)$ only depends on the numerical class $c_1(L)$, so by scaling $\Delta(L)$ makes sense for $\mathbb Q$-line bundles as well. One can furthermore extend by continuity so that $\Delta(L)$ is defined for big $\mathbb R$-line bundles \cite[Theorem B]{LM09}. As pointed out in \cite[\S 1.3]{KL17}, one can alternatively define $\Delta(L)$ for a big $\mathbb R$-line bundle $L$ using
$$
\Delta(L):=\mathrm{closure}\left\{\nu_{Y_\bullet}(D): D\text{ effective $\mathbb R$-divisor, numerically equivalent to } L\right\},
$$
where $\nu_{Y_\bullet}(D)$ is inductively defined in the same way as we did for $\nu_{Y_\bullet}(s)$.

Given $x>0$ such that $L-xY_1$ is a big $\mathbb R$-line bundle, \cite[Proposition 1.6]{KL17} shows that
\begin{equation}
    \label{eq:body-trans}
    \Delta(L-xY_1)=\Delta(L)\cap\{x_1\geq x\}-x \vec e_1.
\end{equation}
This implies that
\begin{equation}
    \label{eq:vol-of-trucated-body}
    \vol_{\mathbb R^n}(\Delta (L)\cap\{x_1\leq x\})=\frac{1}{n!}(\vol(L)-\vol(L-xY_1)).
\end{equation}

Finally, one can read off the positivity information of $L$ from $\Delta(L)$. Recall that the \emph{non-nef locus} of an $\mathbb R$-line bundle $L$ is defined to be
$$
\mathbf{B}_-(L):=\bigcup_{A}\mathbf B(L+A),
$$
where $A$ runs through ample $\mathbb Q$-line bundles. Then $L$ is nef if and only if $\mathbf{B}_-(L)=\emptyset.$

\begin{thm}(\cite[Theorem A]{KL17})\label{thm:KL}
    If $L$ is a big and nef $\mathbb R$-line bundle, then $\Delta_{Y_\bullet}(L)$ must contain the origin. Conversely, given a big  $\mathbb R$-line bundle $L$ and an admissible flag $Y_\bullet$ with $Y_n=\{y\}$, if $\Delta_{Y_\bullet}(L)$ contains the origin, then $y\notin \mathbf{B}_-(L)$. 
\end{thm}

\section{Minimal rational curve with trivial normal bundle}

Throughout this section, we assume that $(X, \eta)$ is twisted Fano pair, 
and $L$ is an ample $\mathbb R$-line bundle and $C\subset X$ is a minimal rational curve. In this part we treat the case where
$$d=-K_X\cdot C=2,$$
so that $C$ has trivial normal bundle by (\ref{eq-nsplit}). The main result of this part is the following.
\begin{thm}
\label{thm:d=2}
    Assume that $C\subset X$ is a minimal rational curve with trivial normal bundle. Then we have $\beta_\eta(X,L)^n\cdot L^n\leq 2n^n$. Moreover if $\eta=2\pi\{\Delta\}+\eta_1$ with $\Delta$ an effective $\IQ$-divisor and $\eta_1\ge 0$ a closed positive (1,1)-current with a nef class $[\eta_1]$, then the equality $\beta_\eta(X,L)^n\cdot L^n= 2n^n$ holds only when $X\cong \mathbb P^1\times\mathbb P^{n-1}$ and $\eta=0$. 
\end{thm}

In what follows we will first derive the inequality $\beta_\eta(X,L)^n\cdot L^n\leq 2n^n$ and then characterize the equality case. 

\subsection{Embedded minimal rational curve with trivial normal bundle}

\begin{prop}\label{prop:trivialNormal}
 Let $(X,\eta)$ be a twisted Fano pair containing an embedded minimal rational curve with trivial normal bundle. Then for any ample $\mathbb R$-line bundle $L$, we have  
 $$
 \beta_\eta(X,L)^n\cdot L^n\leq 2n^n.
 $$
\end{prop}
\begin{proof}
    Let $\pi\colon \hat{X}=\bl_C X \rightarrow X$ be the blowup of $X$ along the curve $C$, with the exceptional divisor $E=\pi^{-1}(C)$. 
    Then $A_X(E)=n-1$. We can rescale $L$ such that $s_{\eta}(X,L)=1$. By Proposition \ref{prop:phi}, we have $\vol(\pi^*L-xE)\geq L^n-2nx^{n-1}$, then 
    \begin{eqnarray*}
        S(L;E)\geq\frac{1}{L^n}\int_0^{\sqrt[n-1]{L^n/2n}} \left(L^n-2nx^{n-1}\right)\, \dif x=\frac{n-1}{n}\left(\frac{L^n}{2n}\right)^{\frac{1}{n-1}}.
    \end{eqnarray*}
    Therefore,
    \begin{eqnarray*}
        \delta_{\eta}(X,L)\leq  \frac{A_\eta(E)}{S(L;E)}\leq  \frac{A_X(E)}{S(L;E)}\leq \frac{n-1}{\frac{n-1}{n}\left(\frac{L^n}{2n}\right)^{\frac{1}{n-1}}}=n \left(\frac{2n}{L^n}\right)^{\frac{1}{n-1}}.
    \end{eqnarray*}
    This implies
    \begin{eqnarray*}
        \delta_{\eta}(X,L)^{n-1}\cdot L^n\leq 2n^{n}.
    \end{eqnarray*}
    Recall that $\beta_{\eta}(X,L)=\min\{s_{\eta}(X,L), \delta_{\eta}(X,L)\}=\min\{1,\delta_{\eta}(X,L)\}$, we conclude that
    \begin{eqnarray*}
        \beta_{\eta}(X,L)^n\cdot L^n \leq \delta_{\eta}(X,L)^{n-1}\cdot L^n\leq 2n^n.
    \end{eqnarray*}
\end{proof}

\subsection{Singular minimal rational curves with trivial normal bundles}
\begin{prop}\label{prop:immersed}
    Suppose $(X, \eta)$ is an $n$-dimensional twisted Fano pair. Assume $f\colon \IP^1 \rightarrow X$ is an immersed minimal rational curve with a non-empty subset of singular points such that $f^*T_{X}=\CO(2)\oplus \CO^{\oplus(n-1)}$, then for any ample $\IR$-line bundle $L$ over $X$, we have $\beta_\eta(X,L)^n\cdot L^n<2n^n$.
\end{prop}
\begin{proof}
We can rescale $L$ so that $s_\eta(X,L)=1$.
By Proposition \ref{prop:phi}, we then have $\vol(L-xE)\geq L^n-\phi(x)$ where
\begin{eqnarray*}
    \phi(x)=
\begin{cases}
    2nx^{n-1}-x^n & \text{if} \,\, x\leq 2\\
    (n-1)x^n+(x-2)^n-n(x-2)x^{n-1}  & \text{if} \,\, x> 2
\end{cases}
\end{eqnarray*}
When $x\leq 2$, we have $\phi(x)\leq 2nx^{n-1}$. When $x>2$, we have $\phi(x)=2nx^{n-1}-x^n+(x-2)^n<2nx^{n-1}$. Let $\psi(x)\coloneqq 2nx^{n-1}$. In particular, we have $\vol(L-xE)>L^n-\psi(x)$. Then, using $A_X(E)\ge A_\eta(E)$, we get: 
\begin{eqnarray*}
    \delta_\eta(X,L)\leq \frac{A_X(E)}{S(L;E)}<\frac{n-1}{\frac{1}{V}\int_0^{\psi^{-1}(L^n)} (L^n-\psi(x))\, \dif x}=\frac{n-1}{\frac{n-1}{n}\left(\frac{L^n}{2n}\right)^{\frac{1}{n-1}}}=\frac{n(2n)^{\frac{1}{n-1}}}{(L^n)^{\frac{1}{n-1}}}.
\end{eqnarray*}
Then we deduce that
\begin{eqnarray*}
    \beta_\eta(X,L)^n\cdot L^n\leq  \delta_\eta(X,L)^{n-1}\cdot L^n <2n^n.
\end{eqnarray*}
\end{proof}

\subsection{The equality case}
Let $C\subseteq X$ be an embedded minimal rational curve with trivial normal bundle. 
In this part we argue that $\beta_{\eta}(X,L)^n\cdot\vol(L)=2n^n$ occurs only when $X\cong\mathbb P^1\times \mathbb P^{n-1}$, hence finishing the proof of Theorem \ref{thm:d=2}. 

Assume, as always, that $n\geq 3$.
Let $\pi:\hat X\to X$ be the ordinary blowup of $X$ along $C$, with exceptional divisor $E$. 
Let
$$\epsilon(L;C):=\max\{x>0:\pi^*L-xE\text{ is nef}\}$$ denote the Seshadri constant $L$ with respect to $C$. It is well known that $\epsilon(L;C)>0$ and that $\pi^*L-xE$ is ample for any $x\in(0,\epsilon(L;C))$ (see e.g. \cite[Lemma 2.9]{B02}).

The key technical result we need is the following characterization of $\epsilon(L;C)$ using the volume function $\vol(\pi^*L-xE)$, whose proof will consist of several steps. 
\begin{prop}
\label{prop:e=vol}
    For any ample $\mathbb R$-line bundle $L$ on $X$ one has
    $$
    \epsilon(L;C)=\max\{a>0:\vol(\pi^*L-xE)=(\pi^*L-xE)^n\text{ for any }0\leq x\leq a\}.
    $$
\end{prop}

When $L$ is an ample $\mathbb Q$-line bundle Proposition \ref{prop:e=vol} follows from \cite[Lemma 3.2]{LM25}, whose proof however cannot be directly extended to the setting of $\mathbb{R}$-line bundles. So here we will use the strategy of \cite[Proposition 4.6]{Zhang22}, which relies on the theory of Okounkov bodies.

Let us first choose a suitable flag on $\hat X$. 
By abuse of notation, we denote by $C$ the section in $E\cong \IP(N_{C/X}^{\vee})$, so that $$E= C\times\IP^{n-2}.$$ Let us also fix a point $y\in C\subseteq E\subseteq \hat{X}$. This already gives us a (partial) flag of subvarieties in $\hat X$. We extend it to a complete flag as follows. Let $E_{y}\cong\IP^{n-2}$ be the fiber in $E$ passing through $y$. We then choose a flag of linear subspaces
$$
E_{y}=Z_0\supset Z_1\supset Z_2...\supset Z_{n-2}=\{y\}
$$
inside $E_{y}$, where $Z_i\cong\mathbb{P}^{n-2-i}$. Let us put
$$
Y_{i+1}=C\times Z_{i},\ i=0,...,n-2,\text{ and }Y_n=\{y\}.
$$
In this way we obtain an admissible flag $Y_\bullet$ of $\hat X$, where $Y_1=E$ and $Y_{n-1}=C$. Moreover,
$$
Y_i\cong \mathbb{P}^1\times \mathbb{P}^{n-1-i}. 
$$
Let $\Delta(L)$ denote the Okounkov body of $\pi^*L$ with respect to this flag. To prove Proposition \ref{prop:e=vol} the key point here is to show that $\Delta(L)$ is of product type near the origin (see Lemma \ref{lem:prod} below) and then we can apply Theorem \ref{thm:KL} as in \cite[Proposition 4.6]{Zhang22} to conclude. 

We begin by observing that  $Y_1\cdot C=0$ (since $C$ is a movable curve in $Y_1$), so that
$$(
\pi^*L-x_1Y_1)\cdot C=L\cdot C=p\ \text{for }x_1\geq 0.
$$
Proceeding inductively, one also has (for $(x_1,...,x_n)\in\mathbb{R}_{\geq 0}$)
\begin{equation}\label{eq:int-C=p}
    \begin{cases}
        ((\pi^*L-x_1 Y_1)|_{Y_1}-x_2Y_2)\cdot C=p,\\
        ...\\
        \bigg( \Big( \big( (\pi^*L - x_1 Y_1)|_{Y_1} - x_2 Y_2 \big)|_{Y_2} \dots \Big)|_{Y_{n-2}} - x_{n-1} Y_{n-1} \bigg) \cdot C = p.
    \end{cases}
\end{equation}

\begin{lem}
    We have that
\begin{equation}
    \label{eq:max-xn<p}
    \max\{x_n: (x_1,...,x_n)\in\Delta(L)\}\leq p.
\end{equation}
\end{lem}

\begin{proof}
    For any effective $\mathbb R$-divisor $D\equiv \pi^*L$ on $\hat X$, assume that $\nu_{Y_\bullet}(D)=(x_1,...,x_n)$. Then \eqref{eq:int-C=p} implies that
    $$
    \mathrm{deg}(((((D-x_1 Y_1)|_{Y_1}-x_2Y_2)|_{Y_2}...)|_{Y_{n-2}}-x_{n-1}{Y_{n-1}})|_{Y_{n-1}})=p.
    $$
    So we deduce that 
    $$
    \ord_{Y_n}(((((D-x_1 Y_1)|_{Y_1}-x_2Y_2)|_{Y_2}...)|_{Y_{n-2}}-x_{n-1}{Y_{n-1}})|_{Y_{n-1}})\leq p,
    $$
    finishing the proof.
\end{proof}

\begin{lem}
    For any $(x_1,...,x_n)\in\mathrm{int}\Delta(L)\cap\{x_1<\epsilon(L;C)\}$ and $1\leq i\leq n-1$, we have that $$((((\pi^*L-x_1 Y_1)|_{Y_1}-x_2Y_2)|_{Y_2}...)|_{Y_{i}}-x_{i+1}{Y_{i+1}})\text{ is ample on }Y_{i}.$$
\end{lem}

\begin{proof}

One can see this inductively. Indeed, $\pi^*L-x_1 Y_1$ is ample on $\hat X$ for $0<x_1<\epsilon(L;C)$. Thus $(\pi^*L-x_1 Y_1)|_{Y_1}$ is ample on $Y_1$. Next, $(\pi^*L-x_1 Y_1)|_{Y_1}-x_2 Y_2$ is a big $\mathbb R$-line bundle on $Y_1$. But $Y_1\cong \mathbb P^1\times\mathbb P^{n-2}$, whose big cone coincides with the ample cone, so $(\pi^*L-x_1 Y_1)|_{Y_1}-x_2 Y_2$ is ample on $Y_1$. Then $((\pi^*L-x_1 Y_1)|_{Y_1}-x_2 Y_2)|_{Y_2}-x_3 Y_3)$ is big on $Y_2$. But since $Y_2\cong \mathbb P^1\times\mathbb P^{n-3}$, whose big cone coincides with the ample cone, we get that $((\pi^*L-x_1 Y_1)|_{Y_1}-x_2 Y_2)|_{Y_2}-x_3 Y_3)$ is ample on $Y_2$. Continuing in this way we conclude.
\end{proof}

Note that on $Y_{i}\cong \mathbb P^1\times \mathbb P^{n-1-i}$ ($1\leq i\leq n-2$) any ample class is generated by $Z_{i-1}$ and $Y_{i+1}$. Using \eqref{eq:int-C=p} we then have that, for any $(x_1,...,x_n)\in\mathrm{int}\Delta(L)\cap\{x_1<\epsilon(L;C)\}$,
\begin{equation*}
\label{eq:ample-sim-pZ}
    ((((\pi^*L-x_1 Y_1)|_{Y_1}-x_2Y_2)|_{Y_2}...)|_{Y_{i}}-x_{i+1}{Y_{i+1}})\sim_\mathbb R t_i Y_{i+1}+p Z_{i-1}
\end{equation*}
for some $t_i>0$. Moreover, on $Y_{n-1}\cong\mathbb P^1$, one has (by \eqref{eq:int-C=p})
\begin{equation}
    \label{eq:ample-sim-pZ-on-C}
     ((((\pi^*L-x_1 Y_1)|_{Y_1}-x_2Y_2)|_{Y_2}...)|_{Y_{n-2}}-x_{n-1}{Y_{n-1}})|_{Y_{n-1}}\sim_\mathbb R pZ_{n-2}=py.
\end{equation}

Now we can show that $\Delta(L)$ is of product type near the origin.
We first treat the case where $L$ is an ample $\mathbb Q$-line bundle.

\begin{lem} Let $L$ be an ample $\mathbb Q$-line bundle.
   For any
$
(x_1,...,x_n)\in\mathrm{int}\Delta(L)\cap\{x_1<\epsilon(L;C)\}
$
one has that
\begin{equation}
    \label{eq:Delta-contain-[0,p]}
    (x_1,...,x_{n-1})\times[0,p]\subset\Delta(L).
\end{equation}
\end{lem}

\begin{proof}
    It amounts to proving that $(x_1,...,x_{n-1},r)\in\Delta(L)$ for any $r\in(0,p)$. To this end, it suffices to treat the case where $(x_1,...,x_{n-1},r)\in\mathbb Q^n$. In what follows $m\in\mathbb N$ is fixed, which is large and sufficiently divisible so that $m (x_1,...,x_{n-1},r)\in\mathbb N^n_{>0}$. Our goal is to construct a section $s_0\in H^0(\hat X,m\pi^*L)$ such that $\nu_{Y_\bullet}(s_0)=m (x_1,...,x_{n-1},r)$. For this we will choose $m$ divisible enough so that Serre vanishing holds for each step of the construction below.

To ease notation, we also introduce ample $\mathbb Q$-line bundles
$
L_i
$
on $Y_i$ ($1\leq i\leq n-1$) in the following inductive manner:
$$
L_1:=(\pi^*L-x_1Y_1)|_{Y_1},\ L_i:=(L_{i-1}-x_i Y_i)|_{Y_{i}}.
$$

    Now, using \eqref{eq:ample-sim-pZ-on-C} and $Y_{n-1}\cong\mathbb P^1$ we can find $s_{n-1}\in H^0(Y_{n-1}, mL_{n-1})$ such that $$\ord_{Y_n}(s_{n-1})=mr.$$ Applying Serre vanishing on $Y_{n-2}$ to the ample line bundle $m(L_{n-2}-x_{n-1}{Y_{n-1}})$ we find that the map 
$
            H^0(Y_{n-2},m(L_{n-2}-x_{n-1}{Y_{n-1}}))\to
            H^0(Y_{n-1}, mL_{n-1})
$
    is surjective. So we obtain
    $$
    \tilde s_{n-1}\in H^0(Y_{n-2},m(L_{n-2}-x_{n-1}{Y_{n-1}}))
    $$
    whose restriction to $Y_{n-1}$ equals $s_{n-1}$. Then the natural inclusion 
$
            H^0(Y_{n-2},m(L_{n-2}-x_{n-1}{Y_{n-1}}))\to
            H^0(Y_{n-2},m L_{n-2}),
$
   maps   $\tilde s_{n-1}$ to $s_{n-2}\in H^0(Y_{n-2},mL_{n-2})$. So we have that
    $$
    \ord_{Y_{n-1}}(s_{n-2})=mx_{n-1}.
    $$

    We can then proceed as above to obtain a sequence of sections
    $
    s_{i}\in H^0(Y_{i}, m L_i)
    $
    with
    $$
    \ord_{Y_{i+1}}(s_i)=mx_{i+1}.
    $$
    So in particular, we have
    $s_{1}\in H^0(Y_{1},m(\pi^*L-x_1Y_1)|{Y_1})$ such that
    $$
    \ord_{Y_{2}}(s_{1})=mx_2.
    $$
    Finally, applying Serre vanishing on $\hat X$ to the ample line bundle $m(\pi^*L-x_1Y_1)$, we obtain $\tilde s_{1}\in H^0(\hat X,m(\pi^*L-x_1Y_1))$ whose restriction to $Y_{1}$ equals $s_{1}$. Under the inclusion $H^0(\hat X,m(\pi^*L-x_1Y_1))\to H^0(\hat X,m\pi^*L)$ we obtain $s_0\in H^0(\hat X,m\pi^*L)$ such that
    $$
    \ord_{Y_1}(s_0)=mx_1.
    $$
   By construction, it holds that $\nu_{Y_\bullet}(s_0)=m (x_1,...,x_{n-1},r)$, so we conclude.
\end{proof}

Combining \eqref{eq:max-xn<p} with \eqref{eq:Delta-contain-[0,p]}, we see that the convex body
$$
\Delta_\epsilon(L):=\{(x_1,...,x_n)\in\Delta(L),\ x_1\leq \epsilon(L;C)\}
$$
is of product type for any ample $\mathbb Q$-line bundle $L$ and can be decomposed as 
$$
\Delta_\epsilon(L)=\Sigma_\epsilon(L)\times[0,p]=\{(x_1,...,x_n): (x_1,...,x_{n-1})\in\Sigma_\epsilon(L),\ x_n\in[0,p]\}
$$
for some convex body $\Sigma_\epsilon(L)\subset\mathbb{R}^{n-1}_{\geq 0}.$
Moreover, one can compute that (by \eqref{eq:vol-of-trucated-body})
$$
\vol_{\mathbb R^n}(\Delta_\epsilon(L))=\frac{\vol(L)-\vol(\pi^*L-\epsilon(L;C)E)}{n!}=\frac{L^n-(\pi^*L-\epsilon(L;C)E)^n}{n!}=\frac{p\epsilon(L;C)^{n-1}}{(n-1)!}.
$$
This implies that
$$
\vol_{\mathbb R^{n-1}}(\Sigma_\epsilon)=\frac{\epsilon(L;C)^{n-1}}{(n-1)!}.
$$

Now let us go back to the case where $L$ is ample $\mathbb R$-line bundle. Then one can pick a sequence of ample $\mathbb Q$-line bundles $L_i$ with $L_i\to L$ in $N^1(X)_{\mathbb R}$ and $L_i-L$ ample, so that $\epsilon(L_i;C)\geq \epsilon(L;C)$. Then in particular the convex bodies $\Delta_{\epsilon_i}(L_i)$ associated to $\pi^*L_i$ does not collapse in $\mathbb R^{n}$ and 
using the continuity of Okounkov bodies in the big cone \cite[Theorem B]{LM09} we obtain the following.

\begin{lem}
\label{lem:prod}
    For any ample $\mathbb R$-line bundle $L$, the convex body\
    $$
\Delta_\epsilon(L):=\{(x_1,...,x_n)\in\Delta(L),\ x_1\leq \epsilon(L;C)\}
$$
is of product type and can be decomposed as 
$$
\Delta_\epsilon(L)=\Sigma_\epsilon\times[0,p]=\{(x_1,...,x_n): (x_1,...,x_{n-1})\in\Sigma_\epsilon(L),\ x_n\in[0,p]\}
$$
for some convex body $\Sigma_\epsilon(L)\subset\mathbb{R}^{n-1}_{\geq 0}.$ 
\end{lem}

Now we are ready to prove Proposition \ref{prop:e=vol}, following the strategy in \cite[Proposition 4.6]{Zhang22}. First, it clear that
$$
\epsilon(L;C)\leq \max\{a>0:\vol(\pi^*L-xE)=(\pi^*L-xE)^n\text{ for any }0\leq x\leq a\}.
$$
Denote the right hand side by $\Lambda_C(L)$. In view of \eqref{eq:body-trans} and Theorem \ref{thm:KL}, it will be enough to show that 
$\Delta(L)$ contains the line segment $[0,\Lambda_C(L)]\times(0,...,0)$.

For any $x\in[0,\Lambda_C(L)]$
we put
$$
\Delta_x(L):=\{(x_1,...,x_n)\in\Delta(L),\ x_1\leq x\}
$$
and
$$
\Sigma_x(L):=\{(x_1,...,x_{n-1}): (x_1,...,x_n)\in\Delta,\ x_1\leq x\}.
$$
Then one has (recall \eqref{eq:vol-of-trucated-body})
\begin{equation}
    \label{eq:vol=x-n-1}
    \vol_{\mathbb R^n}(\Delta_x(L))=\frac{1}{n!}(\vol(L)-\vol(\pi^*L-xE))=\frac{px^{n-1}}{(n-1)!}, x\in[0,\Lambda_C(L)].
\end{equation}
Here $p=L\cdot C$.
Then for $x\in[0,\epsilon(L;C)]$, using the product structure of $\Delta_x$, we obtain that
$$
\vol_{\mathbb R^{n-1}}(\Sigma_x(L))=\frac{x^{n-1}}{(n-1)!},\ x\in[0,\epsilon(L;C)].
$$
As in the proof of \cite[Proposition 4.6]{Zhang22}, the volume growth of this order forces that the convex sets $\Sigma_x(L)$ grow linearly in $x$:
$$
\frac{x}{y}\Sigma_{y}(L)=\Sigma_x(L),\ 0<y\leq x\leq\epsilon(L;C).
$$
Indeed, by convexity and $0\in\Sigma_{\epsilon}(L)$, one has $\frac{x}{y}\Sigma_{y}(L)\supseteq\Sigma_x(L)$. But since both sides have volume $x^{n-1}/(n-1)!$, so the equality follows.
Thus $\Sigma_{\epsilon}(L)$ is a convex cone in $\mathbb R^{n-1}$ with apex at the origin (compare also the philosophy in Cheeger--Colding theory \cite{CC97}: volume cone implies metric cone). Then by convexity we obtain the inclusion of convex bodies:
$$
\Delta_{\Lambda_C(L)}(L)\subseteq \frac{\Lambda_C(L)}{\epsilon(L;C)}\Sigma_\epsilon(L)\times[0,p].
$$
However, both sides have the same volume, thanks to \eqref{eq:vol=x-n-1}. So $\Delta_{\Lambda_C(L)}(L)$ is of product type as well:
$$
\Delta_{\Lambda_C(L)}(L)=\frac{\Lambda_C(L)}{\epsilon(L;C)}\Sigma_\epsilon(L)\times[0,p].
$$

Finally, noting that the segment $[0,\epsilon(L;C)]\times(0,...,0)$ is contained in $\Delta(L)$ by \eqref{eq:body-trans} and Theorem \ref{thm:KL}, then the product structure of $\Delta_{\Lambda_C(L)}(L)$ implies that the segment $[0,\Lambda_C(L)]\times(0,...,0)$ is contained in 
$\Delta(L)$ as well. Pick any $0<\tau<\Lambda_C(L)$, so $\pi^*L-\tau E$ is big. Then we get that
 $y\notin \mathbf{B}_-(\pi^*L-\tau E)$ by \eqref{eq:body-trans} and Theorem \ref{thm:KL} again. Since $y\in E\cong\mathbb P^1\times\mathbb P^{n-2}$ can be an arbitrary point, we obtain that $E\cap\mathbf{B}_-(\pi^*L-\tau E)=\emptyset.$ This implies that $(\pi^*L-\tau E)$ is nef on $\hat X$. Indeed, suppose otherwise that there exists a curve $V\subset \hat X$ such that $(\pi^*L-\tau E)\cdot V<0$. Then $V\subset \mathbf{B}_-(\pi^*L-\tau E)$, and by the projection formula one also has that $E\cap V\not =\emptyset$. Thus $E\cap\mathbf{B}_-(\pi^*L-\tau E)\not=\emptyset$, a contradiction. Sending $\tau\nearrow \Lambda_C(L)$ we conclude Proposition \ref{prop:e=vol}.

\begin{proof}[Completion of Proof of Theorem \ref{thm:d=2}]
Let $C\subset X$ a general minimal rational curve with trivial normal bundle. Then $-K_X\cdot C=2$. 
By Proposition \ref{prop:trivialNormal}, $\beta_\eta(X, L)^n\cdot L^n\le 2n^n$. 
   Moreover from its proof and Proposition \ref{prop:immersed}, when the equality holds, we must have $$\beta_\eta(X,L)=s_\eta(X,L)=\delta_\eta(X,L)=1,\, L^n=2n^n,\  p=L\cdot C=2.$$ And $C$ is an embedded minimal rational curve. By Proposition \ref{prop:e=vol}, we have
    \begin{eqnarray*}
       \epsilon(L;C)\geq \left(\frac{L^n}{np}\right)^{\frac{1}{n-1}}=n.
    \end{eqnarray*}
    We  have the following equality of classes in $H^{1,1}(X, \IR)$: 
    \begin{eqnarray*}
        \pi^*c_1(X, \eta)-n [E]=\left(\pi^*c_1(L)-n[E]\right)+\pi^*\left(c_1(X, \eta)-c_1(L)\right),
    \end{eqnarray*}
    which then implies $\epsilon(c_1(X, \eta);C)\geq n$.
    If $\eta=2\pi\{\Delta\}+\eta_1$ with $\eta_1$ a closed positive (1,1)-current with a nef class, then $(X, \Delta)$ is a log Fano pair and $\epsilon(X, \Delta; C)\ge n=\mathrm{codim}(C)+1$. By 
    Proposition \ref{prop:seshadri-for-log-fano}, $X$ is biholomorphic to $\IP^1\times \IP^{n-1}$. 
    Since $c_1(X, \eta)-\beta_\eta(X, L) c_1(L)$ is a nef class, we have the estimates:
    \begin{equation*}
       \beta_\eta(X, L)^n\cdot L^n\le c_1(X, \eta)^n\le (-K_X)^n=2n^n.
    \end{equation*}
    We claim that the second inequality is strict if $[\eta]\neq 0$. To see this, consider the function $f(t)=([-K_X]-t \frac{[\eta]}{2\pi})^n$ for $t\in [0,1]$. It derivative is equal to:
    \begin{equation*}
        f'(t)=-n \left([-K_X]-t \frac{[\eta]}{2\pi}\right)^{n-1}\cdot \frac{[\eta]}{2\pi}.
    \end{equation*}
    Because the class $[-K_X]-t \frac{[\eta]}{2\pi}$ for $t\in [0,1]$, as a convex linear combination of ample classes, is clearly an ample class, 
    we have $f'(t)<0$ if the pseudoeffective class $[\eta]\neq 0$. After integration, we indeed get $f(1)=c_1(X, \eta)^n< f(0)=(-K_X)^n$. 
\end{proof}

\section{Minimal rational curve with $3\le d\le n-1$}

Let $(X, \eta)$ be a twisted Fano pair and let $C=f(\IP^1)\subset X$ be a minimal rational curve of degree $d=(-K_X)\cdot C$ with $3\le d\le n-1$. In particular, we assume that $n=\dim X\ge 4$.  We prove the following estimates.

\begin{thm}\label{thm:middleDegree}
    For any ample $\IR$-line bundle $L$, we have $\beta_\eta(X,L)^n\cdot L^n <2n^n$ when $3\leq d\leq n-1$.
\end{thm}

\begin{proof}
    Suppose $\pi\colon \hat{X}\rightarrow X$ is the $(1^{\oplus(d-2)},2^{\oplus (n+1-d)})$-weighted blowup along the minimal rational curve $C$ with the exceptional divisor $E$, as discussed in Section \ref{sec-wtbl}. Let $\beta\coloneqq \beta_{\eta}(X,L),\, s\coloneqq s_{\eta}(X,L),\, \delta_\eta\coloneqq \delta_{\eta}(X,L)$. The log discrepancy of $\ord_E$ is given by $A\coloneqq A_X(E)=(d-2)+2(n-d+1)=2n-d$. 
    We denote $V\coloneqq L^n,\, \tau\coloneqq \phi_d^{-1}(V)$, where $\phi_d$ is the function in Proposition \ref{prop:phi}. We introduce $F^{\phi_d}(V)=\frac{1}{V}\int_0^{\tau} \left( V-\phi_d(x)\right)\, \dif x=\tau-\frac{1}{\phi_d(\tau)}\Phi_d(\tau)$, where $\Phi(x)\coloneqq \int_0^x \phi(t)\,\dif t$.  There are two cases to consider: 

\textbf{(Case I)} $\delta_\eta \ge s_\eta$. In this case, we rescale $L$ such that $s_\eta=1$ and hence $\beta_\eta=1$. Recall that $p=L\cdot C\leq d/s_{\eta}=d$. Then, 
\begin{eqnarray*}
    1\leq\delta_{\eta}(X,L)\leq \frac{A_{\eta}(E)}{S(L;E)}\leq 
    \frac{A}{\frac{1}{V}\int_0^{\tau} \left(V-\phi_d(x)\right)\,\dif x}=\frac{A}{F^{\phi_d}(V)}.
\end{eqnarray*}

\textbf{(Case II)} $\delta_\eta<s_\eta$. In this case, we rescale $L$ such that $\delta_\eta(X, L)=\beta_\eta=1$ so that $s_\eta> 1$ and $p=L\cdot f(\IP^1)\le d/s_\eta< d$. We then have:
$$
1=\delta_\eta(X, L)\le \frac{A_\eta(E)}{S(L; E)}\le \frac{A}{F^{\phi_d}(V)}. 
$$
It is easy to verify that the function $V\mapsto F^{\phi_d}(V)$ is increasing (see \cite[Lemma 2.4]{LM25}). So in any case, we get $L^n=V\leq (F^{\phi_d})^{-1}\left(A\right)=\phi_d(T)$ where $T$ is the unique solution to the equation
\begin{eqnarray}\label{eq:tau}
   (T-A)\phi_d(T)=\Phi_d(T).
\end{eqnarray}
By \cite[Proposition 4.8]{LM25}, we get $\beta_{\eta}^n\cdot L^n=L^n \leq \phi_d(T)<2n^n$.
\end{proof}

\section{Minimal rational curves with degree $n$}
In this section, we assume that $X$ is a Fano manifold. 
We have the following classification results. 
\begin{thm}\label{thm-lX}
Let $X$ be a smooth Fano manifold and set 
$$l_X=\min\{\deg(f); f:\IP^1\rightarrow X \text{ is a minimal rational curve on } X\}.$$ Then the following statements are true:
\begin{enumerate}
    \item (\cite{CMSB02}) $l_X=n+1$ if and only if $X\cong \IP^n$.
    \item (\cite{DH17}, \cite[Proposition 2.11]{LM25}) If $n\geq 3$ with $l(X)=n$ and Picard number $\rho(X)=1$, then $(-K_X)^n\leq 2n^n$. Moreover, the equality holds if and only if $X\cong Q^n$. 
    \item (\cite{CD15}) If $n\geq 3$, then $l_X=n$ with Picard number $\rho(X)\geq 2$ if and only if $X$ is the blowup of $\IP^n$ along a smooth codimension two subvariety $A$ of degree $d_A\in \{1,\dots, n\}$ that is contained in a hyperplane. 
\end{enumerate}
\end{thm}
\begin{rmk}
    More generally, it is a natural and interesting question to classify $n$-dimensional uniruled manifolds whose minimal rational curves all have degree at least $n$. 
\end{rmk}
By using this classification result, we verify directly that if $l_X=n$ for a Fano manifold $X$, the $\beta$-volume of any ample $\IR$-line bundle is strictly less than $2n^n$ unless $X$ is itself the quadric hypersurface. 

\begin{prop} \label{prop:d=n}
    Assume that $n\geq 3$. Suppose $X$ is a $n$-dimensional smooth Fano manifold with $l(X)=n$ and $(X,\eta)$ is a twisted Fano pair. Let $L$ be an ample $\IR$-line bundle, then $\beta_{\eta}(X,L)^n\cdot L^n\leq 2n^n$. And the equality holds if and only if $X\cong Q^n$ and $\eta=0$.
\end{prop}
\begin{proof}
    If $\rho(X)\geq 2$, by Theorem \ref{thm-lX}(3), we have $X=\bl_{A}\IP^n$ is the blowup of $\IP^n$ along a complete intersection $A=H_1\cap S_d$ where $H_1 \in |\CO_{\IP^n}(1)|,\, S_d\in |\CO_{\IP^n}(d)|$ with $1\leq d \leq n$. Then for any ample $\IR$-line bundle $L$ on $X$, we will show $\beta_{\eta}(X,L)^nL^n<2n^n$.
    
   (\textbf{Case I}): $d=1$.  Let $\pi\colon X\rightarrow \IP^n$ be the blowup of $\IP^n$ along $A$ and let $E=\pi^{-1}(A)$ be the exceptional divisor. Let $H=\pi^*\CO_{\IP^n}(1)$ be the pullback of a hyperplane section of $\IP^n$. Let $\tilde{H}\sim H-E$ be the strict transform of the hyperplane section $H_1$. The nef cone is given by $\text{Nef}(X)=\IR_{\geq 0}[H]+\IR_{\geq 0}[H-E]$. For any ample $\IR$-line bundle $L$ on $X$, up to $\IR$-linear equivalence, we can assume $L=a(H-E)+bH=(a+b)H-aE$ where $a>0,\, b>0$. We will compute $S(L;E)$ and $S(L;\tilde{H})$ as follows: consider $L-tE=(a+b)H-(a+t)E$ and $L-tE$ is nef if and only if $0\leq t\leq b$. When $0\leq t\leq b$, 
    \begin{eqnarray*}
        \vol(L-tE)&=&\left((a+b)H-(a+t)E \right)^n=(a+b)^n-\sum_{k=2}^n \binom{n}{k}(a+b)^{n-k}(-a-t)^k(k-1)\\
        &=&(b-t)^n+n(a+t)(b-t)^{n-1}.
    \end{eqnarray*}
   Thus, the pseudo-effective threshold $T(L;E)=b$. Consequently, 
    $$
    S(L;E)=\frac{1}{L^n}\int_0^{T(L;E)} \vol(L-tE)\, \dif t=\frac{b\left(2b+(n+1)a\right)}{(n+1)(b+na)}.
    $$
    Next, consider $L-t\tilde{H}\sim (a+b-t)H-(a-t)E$. When $0\leq t \leq a$, divisor $L-t\tilde{H}$ is nef, yielding
    \begin{eqnarray*}
        \vol(L-t\tilde{H})&=&(L-t\tilde{H})^n=\left( (a+b-t)H-(a-t)E\right)^n\\
        &=& (a+b-t)^n -\sum_{k=2}^n \binom{n}{k}(a+b-t)^{n-k}(t-a)^k(k-1)\\
        &=&b^n -n(t-a)b^{n-1}.
    \end{eqnarray*}
    When $t>a$, the Zariski decomposition is given by $L-t\tilde{H}=(a+b-t)H+(t-a)E$, where the nef part is $P(t)=(a+b-t)H$ and the negative part is $N(t)=(t-a)E$, then 
    \begin{eqnarray*}
        \vol(L-t\tilde{H})=(a+b-t)^n H^n=(a+b-t)^n.
    \end{eqnarray*}
    It follows that
    \begin{eqnarray*}
        S(L;\tilde{H})&=&\frac{1}{L^n}\int_0^{T(L;\tilde{H})} \vol(L-t\tilde{H})\, \dif t \\
        &=& \frac{1}{L^n} \int_0^a \left(b^n-n(t-a)b^{n-1} \right)\, \dif t + \frac{1}{L^n} \int_a^{a+b} \left(a+b-t \right)^n\, \dif t\\
        &=& \frac{2(n+1)ab+n(n+1)a^2+2b^2}{2(n+1)(b+na)}.
    \end{eqnarray*}
    Hence, 
   \begin{eqnarray*}
        \delta_{\eta}(X,L)\leq \min\left\{\frac{A(E)}{S(L;E)},\frac{A(\tilde{H})}{S(L;\tilde{H})}\right\}=\min\left\{\frac{(n+1)(b+na)}{b(2b+(n+1)a)}, \frac{2(n+1)(b+na)}{2(n+1)ab+n(n+1)a^2+2b^2}\right\}.
    \end{eqnarray*}
    Setting $t=\frac{a}{b}\in (0,+\infty)$, for $n\geq 2$ we obtain
    \begin{eqnarray*}
        \beta_{\eta}(X,L)^nL^n&\leq& \delta_{\eta}(X;L)^nL^n\leq \min\left\{\frac{(n+1)^n(1+nt)^{n+1}}{(2+(n+1)t)^n},\frac{(2n+2)^n(1+nt)^{n+1}}{(n(n+1)t^2+2(n+1)t+2)^n}\right\}\\
        &\leq& \min\{h_1(t), h_2(t)\},
    \end{eqnarray*}
    where we denote $h_1(t)=\frac{(n+1)^n(1+nt)^{n+1}}{(2+(n+1)t)^n},\, h_2(t)=\frac{(2n+2)^n(1+nt)^{n+1}}{(n(n+1)t^2+2(n+1)t+2)^n}$. By a direct calculation, $h_1(t)$ is strictly increasing when $t>0$ and $h_2(t)$ is strictly decreasing when $t>0$ and $n\geq 2$. Therefore, the function $\min\{h_1(t), h_2(t)\}$ is maximized when $h_1(t)=h_2(t)$, that is, when $t=t_*\coloneqq \sqrt{\frac{2}{n(n+1)}}$. We denote $w=nt_*= \sqrt{\frac{2n}{n+1}}$. Thus, 
    \begin{eqnarray*}
        \beta_{\eta}(X,L)^nL^n \leq \frac{(1+w)^{n+1}}{(\frac{2}{n+1}+\frac{w}{n})^n}=\frac{(1+w)^{n+1}}{(\frac{w^2}{n}+\frac{w}{n})^n}=n^n\frac{1+w}{w^n}<2n^n.
    \end{eqnarray*}
    For the last inequality, we claim that $\frac{1+w}{w^n}<2$. Direct computation shows $w< \sqrt{2}$ for all $n\in \IN^*$ and the sequence $\{w^n\}$ is strictly increasing when $n\geq 3$, yielding $\frac{1+w}{w^n}<\frac{1+\sqrt{2}}{(\sqrt{\frac{3}{2}})^3}<2$, which verifies the claim.

    (\textbf{Case II}): $2\leq d \leq n$. In \cite[Lemma 4.13]{LM25}, it is computed that $(-K_X)^n=\frac{dn^n-(n+1-d)^n}{d-1}<2n^n$ when $d\geq 2$. Since $\beta_{\eta}(X,L)^nL^n$ is rescaling invariant, we may assume $\beta_{\eta}(X,L)=1$. Then in particular, $s_{\eta}(X,L)\geq 1$, so $D\coloneqq -K_X-[\eta]/2\pi-L$ is a nef class. Thus, using $[\eta]$ is pseudoeffective, 
    \begin{eqnarray*}
        L^n =(-K_X-[\eta]/2\pi-D)^n \leq (-K_X-[\eta]/2\pi)^n\leq (-K_X)^n<2n^n.
    \end{eqnarray*}
    
    If $\rho(X)=1$, then by Theorem \ref{thm-lX}(2), we have $(-K_X)^n \leq 2n^n$ and the equality holds if and only if $X\cong Q^n$. By the similar argument as in  (\textbf{Case II}), we may rescale $L$ such that $\beta_{\eta}(X,L)=1$. Then $\beta_{\eta}(X,L)^nL^n\leq (-K_X)^n\leq 2n^n$ and the equality holds if and only if $X\cong Q^n$ and $\eta=0$.
\end{proof}

\section{The proof of main Theorems}
Now we are ready to prove our main theorems. 

\begin{proof}[Proof of Theorem \ref{thm-twFano} and Theorem \ref{thm:rational}]
Let $(X, \eta)$ be a twisted Fano pair.
By the facts recalled at the beginning of Section \ref{sec-wtbl}, $X$ is uniruled and is covered by minimal rational curves.
Consider the invariant $l_X\coloneqq \min\{-K_X\cdot C \mid C\,\text{is minimal rational curve}\}$. 
Choose a general minimal rational curve 
$f: \IP^1\rightarrow X$
in a dominating locally unsplit family whose degree realizes $l_X. $
  If $l_X=2$, by Theorem \ref{thm:d=2} we have $\beta_\eta(X,L)^n\cdot L^n\leq 2n^n$. When $3\leq l_X \leq n-1$, it is proved in Theorem \ref{thm:middleDegree} that $\beta_\eta(X,L)^n\cdot L^n<2n^n$. For the remaining case, we have the length $l_X\geq n$. 

This proves the first part of theorem \ref{thm-twFano}. Now we assume furthermore that $X$ is a Fano manifold. 
The equality $\beta_\eta(X, L)^n \cdot L^n=2n^n$ can only occur when $l_X=2$ or $l_X=n$. 
When $l_X=2$, the statement follows from the equality case from Theorem \ref{thm:d=2}. 
By the result in Theorem \ref{thm-lX}, we have $l_X=n$ if and only if $X\cong Q^n$ or $X\cong \bl_{A}\IP^n$, where $A=H\cap S_d,\, H\in |\CO_{\IP^n}(1)|$ is a hyperplane section and $S_d\in |\CO_{\IP^n}(d)|$ is a degree $d$ hypersurface of $\IP^n$. By Proposition \ref{prop:d=n}, when $l(X)=n$, we have $\beta_\eta(X,L)^n\cdot L^n\leq 2n^n$ and the equality holds only if $X\cong Q^n$ and $\eta=0$.     
\end{proof}

\begin{cor}
    Suppose $X$ is an $n$-dimensional smooth K-unstable Fano manifold, then 
     $$
     \delta(X,-K_X)^n\cdot(-K_X)^n <2n^n.
     $$
 \end{cor}
 \begin{proof}
     When we choose $L=-K_X$, we get $\beta(X,-K_X)=\min\{1,\delta(X,-K_X)\}=\delta(X,-K_X)$ since $X$ is K-unstable. Then by Theorem \ref{thm:rational} and the fact that $\IP^n,\, \IP^1\times \IP^{n-1},\, Q^n$ are K-semistable, we get  $\delta(X,-K_X)^n\cdot(-K_X)^n <2n^n$.
\end{proof}
 \begin{rmk}
     It would be interesting to determine the sharp upper bound for the $\delta$-volume $\delta(X,-K_X)^n\cdot(-K_X)^n$ among all $n$-dimensional K-unstable Fano manifolds.
 \end{rmk}
\begin{thm}[=Theorem \ref{thm:gap}]
Let $(X^n,\omega)$ be an $n$-dimensional K\"ahler manifold with $\ric(\omega)\geq (n+1)\omega$. Assume $X\ncong \IP^n$, then we have 
$$
\vol(X,\omega)\leq\frac{2n^n}{(n+1)^n}\vol(\mathbb P^n,\omega_{FS}).
$$     
And the equality holds if and only if $(X,\omega)$ is biholomorphically isometric to $Q^n$ or $ \IP^1\times \IP^{n-1}$ equipped with their standard K\"ahler--Einstein metrics with Einstein constant $(n+1)$.
\end{thm}
\begin{proof}
    Let $L$ be the ample $\mathbb R$-line bundle corresponding to the class $\frac{1}{2\pi}[\omega]$, then $\beta(X,L)\geq n+1$. Then by Theorem \ref{thm:rational}, we get $\vol(L)\leq \frac{2n^n}{(n+1)^n}$. Thus, 
    \begin{eqnarray*}
       \vol(X,\omega)= \int_X \omega^n = (2\pi)^n \cdot \vol(L)\leq (2\pi)^n\frac{2n^n}{(n+1)^n}=\frac{2n^n}{(n+1)^n}\vol(\mathbb P^n,\omega_{FS}).
    \end{eqnarray*}
    By Theorem \ref{thm:rational} the equality holds only if $X\cong Q^n$ or $X\cong \IP^1\times \IP^{n-1}$. 
    
    It remains to show that in the equality case $(X,\omega)$ is actually isometric to $Q^n$ or $ \IP^1\times \IP^{n-1}$ equipped with their K\"ahler--Einstein metrics.

    \textbf{Case 1:} $X= Q^n.$

Consider $\tilde \omega:=(n+1)\omega$ and $\tilde \xi:=\frac{1}{2\pi}[\tilde \omega]$. Note that $\tilde \xi^n=2n^n=c_1(Q^n)^n$. Using that $Q^n$ has Picard number one when $n\geq 3$, we deduce that $\tilde \xi=c_1(Q^n)$. On the other hand, one has $\Ric(\tilde \omega)\geq \tilde \omega$, so the cohomological relation forces that $\Ric(\tilde \omega)=\tilde\omega$. Namely, $\tilde \omega$ is a K\"ahler--Einstein metric. By uniqueness of such metrics \cite{BM87} we conclude.

    \textbf{Case 2:} $X= \IP^1\times \IP^{n-1}.$

Consider again $\tilde \omega:=(n+1)\omega$ and $\tilde \xi:=\frac{1}{2\pi}[\tilde \omega]$.
    Let $p_1: X\to \IP^1$ and $p_2: X\to \IP^{n-1}$ be the natural projection maps. Set $H_1:=p_1^*\mathcal{O}_{\IP^1}(1)$ and $H_2:=p_2^*\mathcal{O}_{\IP^{n-1}}(1)$. Then $\tilde \xi$ is of the form $\tilde \xi=a H_1+b H_2$ for some $a>0$ and $b>0$, so that $\tilde \xi^n=nab^{n-1}.$ Note that $c_1(X)=2H_1+nH_2$ and $\Ric(\tilde \omega)\geq \tilde \omega$. Thus $(2H_1+nH_2)-(a H_1+b H_2)$ is nef, which implies that $a\leq 2$ and $b\leq n$. Then the volume equality $\tilde \xi^n=2n^n$ forces that $a=2$ and $b=n$, so that $[\Ric(\tilde \omega)]=[\tilde \omega]$. This further ensure that $\Ric(\tilde \omega)= \tilde \omega$ and we conclude by \cite{BM87} again.
    
\end{proof}
\section{Toric log Fano pairs}

In this section, we resolve Andreasson-Berman's "logarithmic gap hypothesis" (\cite{AB25}). The log pair $(X,\D)$ is a toric log Fano pair if $X$ is a normal toric variety and $\D$ is an effective $\IT$-invariant $\IQ$-divisor, and $(X,\D)$ is klt with $-K_X-\D$ is ample $\IQ$-Cartier.
\begin{thm}\label{thm:toricSmooth}
    Suppose $(X,\D)$ is a K-semistable toric log Fano pair, and $X$ is smooth. If $X\ncong \IP^n$, then $(-K_X-\D)^n\leq (-K_{\IP^1\times \IP^{n-1}})^n=2n^n$. The equality holds if and only if $X\cong \IP^1\times \IP^{n-1}$ and $\Delta=0$. 
\end{thm}
\begin{proof}
    By \cite[Corollary 2.6]{CFH14},
    there exists an open dense subset $U$ of $X$ that is isomorphic to $\IP^{p+1}\times (\IC^*)^{n-p-1}$ as toric varieties with $p+1\in \{1,2,\cdots, n\}$. 
    In particular, any fiber $Z\cong \IP^{p+1}$ of the projection $U\rightarrow (\IC^*)^{n-p-1}$ has a trivial normal bundle $N_{Z/X}=\CO_Z^{\oplus r}$ where $r=n-1-p$. Since we assume $X\ncong \IP^n$, then $r\in \{1,2,\cdots, n-1\}$. By the similar argument with \cite[Proposition 3.1]{LM25}, we have 
    \begin{eqnarray*}
        \vol_{\hat{X}}(\pi^*L-xE) &\geq&  L^n - \mathop{\limsup}_{k\rightarrow +\infty}\sum_{j=0}^{xk-1} \frac{n!}{k^n} h^0(Z,kL|_Z\otimes \CI_Z^{j}/\CI_Z^{j+1})\\
        &=& L^n-\binom{n}{r}(L|_Z)^{n-r}\cdot x^r.
    \end{eqnarray*}
    Then,
    \begin{eqnarray*}
        S(L;E)=\frac{1}{L^n}\int_0^{+\infty} \vol_{\hat{X}}(\pi^*L-xE)\, \dif x \geq \frac{1}{L^n}\int_0^{\varepsilon} \left(L^n-\binom{n}{r}(L|_Z)^{n-r}\cdot x^r\right)\, \dif x=\frac{r}{r+1}\varepsilon,
    \end{eqnarray*}
    where $\varepsilon=(\frac{L^n}{\binom{n}{r}\cdot (L|_Z)^{n-r}})^{1/r}$. Then, 
    \begin{eqnarray*}
        \delta(X,L) \leq \frac{A_{(X,\D)}(E)}{S(L;E)}\leq \frac{r}{\frac{r}{r+1}\varepsilon}=\frac{r+1}{\varepsilon}.
    \end{eqnarray*}
    This implies 
    \begin{eqnarray*}
        \delta(X,L)^r\cdot L^n \leq \binom{n}{r}(r+1)^r (L|_Z)^{n-r}.
    \end{eqnarray*}
    Recall that $(X,\D)$ is K-semistable implies $\delta(X,L)\geq 1$ and $L=-K_X-\D$, we get
    \begin{eqnarray} \label{ineq:Z}
        (-K_X-\D)^n \leq \binom{n}{r}(r+1)^r ((-K_X-\D)|_Z)^{n-r}.
    \end{eqnarray}
    We claim that $Z \nsubseteq \Supp(\D)$. By the construction of \cite[Corollary 2.5]{CFH14},  $Z$ passes through the identity $e\in \IT=(\IC^*)^n$ while $\Supp(\D)\cap \IT =\emptyset$, which shows the claim. Furthermore, $(-K_X)|_Z=-K_Z$ since $N_{Z/X}$ is trivial and $\D|_Z$ is an effective $\IQ$-divisor on $Z$, so 
    \begin{eqnarray*}
        \vol_Z((-K_X-\D)|_Z)=(-K_Z-\D|_Z)^{n-r}\leq (-K_Z)^{n-r}.
    \end{eqnarray*}
    Following the inequality (\ref{ineq:Z}), we get
    \begin{eqnarray*}
        (-K_X-\D)^n \leq \binom{n}{r}(r+1)^r ((-K_X-\D)|_Z)^{n-r} \leq  \binom{n}{r}(r+1)^r (-K_Z)^{n-r}=(-K_{\IP^r\times \IP^{n-r}})^n.
    \end{eqnarray*}
    By \cite[Lemma 3.4]{LM25}, $c_r=(-K_{\IP^r\times \IP^{n-r}})^n=\binom{n}{r}(r+1)^r (n-r+1)^{n-r}$ attains maximum when $r=n-1$ or $r=1$. Thus, $(-K_X-\D)^n \leq (-K_{\IP^1\times \IP^{n-1}})^n = 2n^n$.

 When the equality $(-K_X-\D)^n=2n^n$ holds, we conclude that $r=n-1$ or $r=1$ and 
 $\vol_{\hat{X}}(\pi^*L-xE)=(\pi^*L-xE)^n$ when $x\in [0, \epsilon]$ where 
 $$\varepsilon=\left(\frac{(-K_{\IP^r\times \IP^{n-r}})^n}{\binom{n}{r}(-K_{\IP^{n-r}})^{n-r}}\right)^{\frac{1}{r}}=r+1.$$
 So by \cite[Lemma 3.2]{LM25} the Seshadri constant $\epsilon(-K_X-\Delta; Z)\ge r+1$ which is greater than the codimension of $Z$. Note that $Z$ is isomorphic to $\IP^1$ if $r=n-1$ and to $\IP^{n-1}$ if $r=1$. 
 In both cases, by Proposition \ref{prop:seshadri-for-log-fano}, $X$ is biholomorphic to $\IP^1\times \IP^{n-1}$. Finally when $X\cong \IP^1\times \IP^{n-1}$ and $\Delta\neq 0$, it is easy to see that $\vol(-K_X-\Delta)=(-K_X-\Delta)^n$ is strictly less than $2n^n=\vol(-K_X)$. So we get the last statement. 

\end{proof}
\begin{rmk}
    Regarding the classification argument in the proof of \cite[Proposition 3.4]{AB25}, we note that the underlying variety of a log Fano pair need not itself be Fano, even when the boundary is effective. For instance, consider the Hirzebruch surface $X=F_2=\IP(\CO_{\IP^1}\oplus \CO_{\IP^1}(2))$ and let $S$ be its negative section. For any rational number $\varepsilon\in (0,1)$, the pair $(X,\varepsilon S)$ is a log Fano pair with effective boundary divisor $\D=\varepsilon S$ while $X$ is not Fano. 
\end{rmk}
\begin{prop}\label{prop:toricSingular}
    Suppose $(X,\D)$ is a K-semistable toric log Fano pair such that $X$ is a singular toric variety, then $(-K_X-\D)^n<2n^n$.
\end{prop}
\begin{proof}
    We pick a $\IT$-invariant non-smooth point $x\in X$, then $(X,x)$ is a toric singularity. By Liu's local-to-global volume comparison in the log setting (\cite[Proposition 4.6]{LL19}, see also the definition of local volume $\hvol(x;X,\D)$ within the reference) and by logarithmic version of toric ODP (Proposition \ref{prop:ODP}),
    \begin{eqnarray*}
        (-K_X-\D)^n \leq \left(\frac{n+1}{n} \right)^n\cdot \hvol(x;X,\D)\le\left( \frac{n+1}{n}\right)^n \cdot 2(n-1)^n<2n^n.
    \end{eqnarray*}
\end{proof}
\begin{proof}[Proof of Theorem \ref{thm:toricPair}]
    Directly follows from Theorem \ref{thm:toricSmooth} and Proposition \ref{prop:toricSingular}.
\end{proof}
\appendix

\section{local volume of toric log singularities}\label{app-vol}
In this appendix, we prove the following logarithmic version of Moraga-S\"{u}ss' estimate for toric singularities.  
\begin{prop}\label{prop:ODP}
    Let $x\in (X, \Delta)$ be an $n$-dimensional $\IQ$-Gorenstein toric klt log singularity. Assume that $x$ is not a regular point. 
    Then $\hvol(x; X, \Delta)\le 2(n-1)^n$. 
\end{prop}
This can be proved using the same method of 
of Moraga-S\"{u}ss \cite{MS24}, together with the observation that the inclusion of extra boundary toric divisor actually improves the estimate of the local volume. For the reader's convenience, we sketch the proof. Of course, if $X$ itself is $\IQ$-Gorenstein, the estimate immediately follows from the result of Moraga-S\"{u}ss. 
\begin{proof}
    Let $\sigma\subset N\otimes \IR\cong \IR^n$ be the cone with $X_\sigma=X$. Let $\{\rho\}$ be the set of 1-dimensional rays of $\sigma$ and $\sigma(1):=\{v_\rho\}$  the primitive lattice point on $\rho$. Assume $\Delta=\sum_\rho c_\rho \Delta_\rho$ with $c_\rho\in [0,1)$ where $\Delta_\rho$ denotes the toric divisor on $X=X_\sigma$ corresponding to the ray $\rho$. Let $\sigma^\vee\subset M\otimes \IR$ be the dual cone of $\sigma$.      
    Since $(X, \Delta)$ is $\IQ$-Gorenstein, the exists
    $u\in \sigma^\vee\cap \IQ^n$ such that $\langle u, v_\rho\rangle=1-c_\rho$.
    If we set $v'_\rho=\frac{v_\rho}{1-c_\rho}$, then $\langle u, v'_\rho\rangle=1$. Set $\mathfrak{v}=\{v_\rho; \rho\in \sigma(1)\}$ and $\mathfrak{v}'=\{v'_\rho; \rho\in \sigma(1)\}$. 
    Let $P'=\mathrm{conv}(\mathfrak{v}')$ denote the convex hull of $v'_\rho$. Then $P'$ is an $(n-1)$-dimensional polytope contained in the hyperplane $\langle u, \cdot\rangle= 1$. 
    Let $\ell\in \IQ_{\geq1}$ be the rational number such that $\ell u$ is a primitive lattice point. Up to unimodular equivalence, we can assume that $\ell u=(0,1)\in \IZ^{n-1}\times \IZ$ and $\sigma=\IR_{\ge 0}\cdot (P'\times \{\ell\})$. With this assumption, we write $v_\rho=(w_\rho, y_\rho)\in \IZ^{n-1}\times \IZ_{>0}$ which then satisfies $\frac{y_\rho}{\ell(1-c_\rho)}=1$ for any $\rho\in \sigma(1)$. 

    Recall that the Santal\'{o} point $\chi=\chi(P')$ of $P'$ is the point in the interior of $P'$ that minimizes the volume of the dual polytope $P'^y:=(P'-y)^\vee$ among all points $y$ in the interior of $P'$. 
    By the same proof as in \cite[Lemma 4.2]{MS24}, we get:  
    \begin{equation*}
        \hvol(x;X, \Delta)=\frac{(n-1)!}{\ell}\vol(P'^\chi). 
    \end{equation*} 
    To estimate this volume, we consider the Mahler volume (also called the product volume):
    \begin{equation*}
        \mathcal{MV}(P')=\vol(P')\vol(P'^\chi).
    \end{equation*}
     
    Denote by $|\mathfrak{v}|=|\mathfrak{v}'|$ the number of elements of $\mathfrak{v}$ or $\mathfrak{v}'$.  
    We consider three cases:
    \begin{enumerate}
        \item $|\mathfrak{v'}|=(n-1)+1=n$, then $P'$ is an $(n-1)$-dimensional simplex contained in the hyperplane $\langle u, \cdot\rangle=1$. 
         Because the Mahler volume 
         is invariant under affine transformations, we know that 
         \begin{equation*}
             \mathcal{MV}(P')=\frac{n^n}{((n-1)!)^2}
         \end{equation*}
         which is the Mahler volume of the standard simplex. 
         Denote by $C$ (resp. $C'$) the convex hull of $\{v_\rho; \rho\in \sigma(1)\}\cup \{0\}$ (resp. $\{v'_\rho; \rho\in \sigma(1)\}\cup \{0\}$).
         Then $|v'_\rho|=\left|\frac{v_\rho}{1-c_\rho}\right|\ge |v_\rho|$ and hence $C\subset C'$. So we have $\vol(C)\le \vol(C')= \frac{\ell}{n}\vol(P')$.      
         Note that $C$ is $n$-dimensional lattice polytope that generates the cone $\sigma$, i.e. $\sigma=\IR_{\ge 0}C$.  
         Since $x\in X$ is not a regular point, $\sigma$ is not unimodular equivalent to the cone $\IR_{\ge 0}^n$. We get $\vol(C)\ge \frac{2}{n!}$.
         So we get $\vol(P')\ge \frac{2}{(n-1)!\ell}$ and hence
        \begin{equation*}
            \hvol(x; X, \Delta)=\frac{\mathcal{MV}(P')}{\vol(P')}\frac{(n-1)!}{\ell}\le \frac{n^n}{2}<2(n-1)^n. 
        \end{equation*}
        \item $|\mathfrak{v'}|=(n-1)+2=n+1$. 
         Then by Radon's theorem, 
        $P'$ has a $(p,q)$-partition with $p+q=n-1$. This means that there exists a partition of $\{v'_\rho; \rho\in \sigma(1)\}$ into two non-empty set $A'_1$ and $A'_2$ such that $Q'_i=\mathrm{Conv}(A'_i)$, $i=1,2$ are simplices of dimension $p$ and $q$ respectively, which intersect at a unique point (called the Radon point). Let $p>0$ be the minimal number such that $P$ admits a $(p,q)$-partition. 
         By \cite[Theorem 5.1]{AFZ19}, we know that
        \begin{equation*}
           \mathcal{MV}(P')\le \frac{(p+1)^{p+1}(q+1)^{q+1}}{(n-1)!p!q!}.
        \end{equation*}     
        By the same argument as in \cite[Lemma 4.7]{MS24}, we get a triangulation of $P'$ into $(p+1)$ full-dimensional simplices whose vertices are contained in $\{v'_\rho; \rho\in \sigma(1)\}$. 
      Since the volume of each one of these simplices is at least the volume of the standard lattice simplex, we get:
        $\vol(P')\ge \frac{p+1}{\ell(n-1)!}$ and hence:
        \begin{eqnarray*}
            \hvol(x; X, \Delta)&=&\frac{\mathcal{MV}(P')}{\vol(P')}\frac{(n-1)!}{\ell}\le (n-1)!\frac{(p+1)^p(q+1)^{q+1}}{p!q!}\\
            &\le& 2(n-1)^{n-1}(q+1)\le 2(n-1)^n. 
        \end{eqnarray*}
        The second to last inequality follows from \cite[Lemma 3.4]{LM25} (see also \cite[Lemma 4.10]{MS24}). The last inequality follows because $q=n-1-p\le n-2$. 
        \item $|\mathfrak{v}'|\ge n+2$. There exists a polytope $Q'$ spanned by $n+1$ points of $\mathfrak{v}'$. If $z\in Q'$ is the Santal\'{o} point of $Q'$, then $\vol(P'^\chi)\le \vol(P'^z)\le \vol(Q'^z)$. Then the conclusion follows from the estimate in the second case.        
    \end{enumerate}
\end{proof}

\section{Log Fano pairs with large Seshadri constants}

Let $(X, \Delta)$ be a log Fano pair with $X$ smooth. Let $Z\subset X$ be a smooth codimension $r$ variety with a trivial normal bundle. 
Consider the blowup of $X$ along $Z$: $\pi: X'=\mathrm{Bl}_ZX\rightarrow X$. Let $E$ denote the inverse image $\pi^{-1}(Z)$ and $\Delta'$ be the strict transform of $\Delta$. For the main result in this appendix, we will assume that the following two conditions hold:
\begin{enumerate}
    \item $Z\not\subseteq \mathrm{Supp}(\Delta)$ and $(Z, \Delta|_Z)$ is a log Fano pair. 
    \item $(X', \Delta')$ has klt singularities. 
\end{enumerate}
For our application, we are interested in two cases:
\begin{itemize}
    \item (Case 1) $Z$ is a general member of a covering family of embedded minimal rational curves of degree $-K_X\cdot Z=2$. 
    \item (Case 2) There exists a Zariski open set $U\subset X$ and a smooth fibration $\pi: U\rightarrow Y$ with 1-dimensional base such that $Z$ is a general fiber. 
\end{itemize}
In case 1, since $Z$ is a general smooth curve in a covering family, we can assume that $Z$ intersects transversely with $\Delta$ at smooths point of $\Delta$. In case 2, $Z$ the blowup morphism is the identity and $(X', \Delta')=(X, \Delta)$. In both cases, one can easily verify the two conditions above. 

We have the following logarithmic  version of results about Fano varieties with large Seshadri constants from \cite{BS09, LZ18, LM25}. See Remark \ref{rmk-logSe}. 

\begin{prop}
\label{prop:seshadri-for-log-fano}
    Let $(X, \Delta)$ be a log Fano pair with $X$ smooth, and let $Z$ be a smooth subvariety with a trivial normal bundle that satisfies the two conditions above.
    If the Seshadri constant $\epsilon(-K_X-\Delta; Z)$ is at least $r+1$ where $r$ is the codimension $Z$, then $X\cong \IP^{r}\times Z$. 
\end{prop}

\begin{proof}[Proof of Proposition \ref{prop:seshadri-for-log-fano}]
     We continue to use the notation at the beginning of this appendix and use a similar line of argument as in \cite{LZ18} with an important input from \cite{JLR26}.  
    
    Set $L=-K_X-\Delta$ and $D=\pi^*(-K_X-\Delta)-\epsilon E$ where $\epsilon=\epsilon(L; Z)$. Denote by $\Delta'$ the strict transform of $\Delta$ under $\pi$. Recall that we assumed that $(X', \Delta')$ has klt singularities. Note that
    \begin{equation*}
        D-(K_{X'}+\Delta')=2\left(\pi^*(-K_X-\Delta)-\frac{\epsilon+a}{2} E\right)
    \end{equation*}
    is ample where  
    $a=a(E; X, \Delta)=a(E; X)-\ord_E(\Delta)\le r-1<r+1\le \epsilon$.  
    By the basepoint-free theorem for $\IR$-divisors (see \cite[Theorem 17.1]{Fuj11}), we know that $D$ is semiample. In other words (see \cite[Lemma 4.13]{Fuj11}), there exists a morphism $\Phi: X'\rightarrow Y$ and an ample divisor $A=\sum_i c_i A_i$ with $c_i>0$ and $A_i$ ample Cartier divisors such that $D=\Phi^*A$. We claim that $\Phi|_E: E\rightarrow Y$ is a closed embedding. To see this, first assume that $D$ is a $\IQ$-Cartier divisor. Then for $m\ge 1$ such that $mD$ is Cartier, we have:
    \begin{equation*}
        m D-E-(K_{X'}+\Delta')=(m+1)\left(\pi^*(-K_X-\Delta)-\frac{m\epsilon+1+a}{m+1}E\right)
    \end{equation*}
    is ample since $(m\epsilon+1+a)/(m+1)<\epsilon$. By the Kawamata-Viehweg vanishing theorem, $H^1(X', m D-E)=0$. Therefore the natural map $H^0(X', mD)\rightarrow H^0(E, m D|_E)$ is surjective for all sufficiently large and divisible $m$. Moreover, since $E\cong Z\times \IP^{r-1}$ and $\CO_E(-E)=\CO_{\IP^{r-1}}(1)$, it is easy to see that $D|_E$ is ample. So $\Phi|_E: E\rightarrow Y$ is a closed embedding.   
    If $D$ is just an $\IR$-divisor, we can choose $0<\delta_i\ll 1$ such that $0<c_i-\delta_i\in \IQ$ and $D'=\Phi^*\left(\sum_i (c_i-\delta_i)A_i\right)$ satisfies the property that $D'-E-(K_{X'}+\Delta')$ remains ample. Then for sufficiently large and divisible $m$, 
    $$
    mD'-E-(K_{X'}+\Delta')=(m-1)D'+\left(D'-E-(K_{X'}+\Delta')\right)
    $$
    is ample and we can apply the above argument to $D'$ and we conclude that $\Phi|_E: E\rightarrow Y$ is indeed a closed embedding with $E\cong \IP(N^\vee_{Z/X})\cong Z\times \IP^{r-1}$. Moreover, we get $-(K_{X'}+\Delta') \sim_{\Phi,\IR} \lambda E$ with $\lambda=\epsilon-a\ge (r+1)-(r-1)=2$. 

    Suppose $\Phi$ contracts a curve $C$ to a point $y\in Y$. Since $\Phi|_E$ is a closed embedding, the scheme-theoretic intersection $E\cap \Phi^{-1}(y)$ is a reduced closed point. This implies $\Phi^{-1}(y)=C$ (i.e. nontrivial each fiber is 1-dimensional) and $C$ intersects $E$ transversely at a single point. In particular $C\cdot E=1$. 
    So we get $-(K_{X'}+\Delta')\cdot C=\lambda E\cdot C=\lambda\ge 2$. Thanks to the recent result about optimal bounds bend-and-break applied to extremal contractions (\cite[Theorem 1.2]{JLR26}), if $\Phi$ is birational, then $-(K_{X'}+\Delta')\cdot C<\dim C+1=2$ and we have a contradiction. 
    So we conclude that $\Phi$ must be of fiber type and $\dim Y=n-1$. 
     
    The rest of argument follows \cite{LZ18}. 
    Recall that from above $\Phi|_E: E\rightarrow Y$ is a closed embedding. We conclude that $Y\cong E\cong \IP(N^\vee_{Z/X})\cong Z\times \IP^{r-1}$. A general fiber of $\Phi$, being $K_{X'}+\Delta'$-negative, is a smooth rational curve. For any fiber $C=\Phi^{-1}(y)$ as above, consider the normal surface $S$ from above. Then the morphim $\Phi|_S: S\rightarrow \Phi(S)$ is of fiber type. All fibers of $\Phi|_S$ are generically reduced and irreducible. 
    The generic fiber of $\Phi|_S: S\rightarrow \Phi(S)$ is isomorphic to $\IP^1$. 
    By \cite[Lemma 6]{LZ18}, we have $C\cong \IP^1$. This implies that $\Phi: X'\rightarrow Y$ is a smooth $\IP^1$-fibration. 

    Note that $s=\Phi|_E^{-1}: Y\rightarrow E$ gives a section of $\Phi$. Then there exists a rank 2 vector bundle $\CE$ over $Y$ such that $X'=\IP_Y(\CE)$. The section $E$ corresponds to a surjection $\CE\rightarrow \CL$ for some line bundle $\CL$ on $Y$. Denote the kernel of this surjection by $\CK$. Then $\CO_Y(-1)=s^*N_{E/X'}\cong \CL\otimes \CK^{-1}$. Since $\CE$ is determined up to twisting of a line bundle. We may assume $\CK=\CO_Y$ and $\CL=\CO_Y(-1)=p_2^*\CO_{\IP^{-1}}(-1)$ with the projection $p_2: Y\rightarrow \IP^{r-1}$. Moreover, we have the following short exact sequence:
    \begin{equation*}
        0\rightarrow \CO_Y\rightarrow \CE\rightarrow \CO_Y(-1)\rightarrow 0.
    \end{equation*}
    Since $(Z, \Delta|_Z)$ is assumed to be a log Fano pair, we have $H^1(Z, \CO_Z)=0$ by using Kawamata-Viehweg vanishing theorem. So we get $\mathrm{Ext}^1(\CL, \CO_Y)=H^1(Z\times \IP^{r-1}, p_2^*\CO_{\IP^{r-1}}(1))=H^1(Z, \CO_Z)\otimes H^0(\IP^{r-1}, \CO_{\IP^{r-1}}(1))=0$ and hence the above exact sequence actually splits. Therefore $\CE\cong \CO_Y\oplus \CO_Y(-1)$ and $X'=\IP(\CO_Y\oplus \CO_Y(-1))\cong \IP(\CO_{\IP^{r-1}}\oplus\CO_{\IP^{r-1}}(-1))\times Z=\mathrm{Bl}_p\IP^r\times Z$ is isomorphic to $\IP^r\times Z$ blowup along $\{p\}\times Z$. So $X\cong \IP^r\times Z$. 
\end{proof}

\begin{rmk}\label{rmk-logSe}
In contrast to the absolute case (i.e. without boundary divisor) in \cite{LZ18} and \cite[Proposition 2.8]{LM25}, the above result is not true if we just assume that the Seshadri constant $\epsilon(-K_X-\Delta; Z)$ is greater than the codimension $r$ of $Z$ in $X$. 
Indeed, one has the following counterexample.
Consider the Hirzebruch surface $X=\IP(\CO_{\IP^1}(1)\oplus \CO_{\IP^1})$ with the natural projection $\pi: X\rightarrow \IP^1$, and let $E$ be the section with self-intersection $-1$.
Set $\Delta=c E$ with $c\in (0,1)$ and $Z=\pi^{-1}(\{pt\})\cong \IP^1$ a fiber. Then it is easy to show that $(X, \Delta)$ is a log Fano pair and the Seshadri constant $\epsilon(-K_X-\Delta; Z)=1+c>1=\mathrm{codim}(Z)$.  But $X$ is not the product $\IP^1\times \IP^1$. The cause for this difference is that, in the proof for the absolute case, one can show by using \cite[Lemma 5]{LZ18} that $-K_{X'}\cdot C\le 1$ compared to the inequality $-(K_{X'}+\Delta')\cdot C<2$ used in the above proof.

\end{rmk}

\bibliographystyle{alpha}
\bibliography{ref}

@article {AFZ19,
    AUTHOR = {Alexander, Matthew and Fradelizi, Matthieu and Zvavitch, Artem},
     TITLE = {Polytopes of maximal volume product},
   JOURNAL = {Discrete. Comput. Geom.},
  FJOURNAL = {Discrete and Computational Geometry},
    VOLUME = {62},
      YEAR = {2019},
     PAGES = {583-600},
      ISSN = {},
   MRCLASS = {},
  MRNUMBER = {},
       DOI = {},
       URL = {https://doi.org/10.1007/s00454-019-00072-3},
}

@article {AB24,
    AUTHOR = {Andreasson, Rolf and Berman, Robert J.},
     TITLE = {Sharp bounds on the height of {K}-semistable {F}ano varieties
              {I}, the toric case},
   JOURNAL = {Compos. Math.},
  FJOURNAL = {Compositio Mathematica},
    VOLUME = {160},
      YEAR = {2024},
    NUMBER = {10},
     PAGES = {2366--2406},
      ISSN = {0010-437X,1570-5846},
   MRCLASS = {14G40 (11G50 14J45 32Q20 53C25)},
  MRNUMBER = {4806901},
       DOI = {10.1112/S0010437X2400736X},
       URL = {https://doi.org/10.1112/S0010437X2400736X},
}

@article {AB25,
    AUTHOR = {Andreasson, Rolf and Berman, Robert J.},
     TITLE = {Sharp bounds on the height of {K}-semistable {F}ano varieties
              {II}, the log case},
   JOURNAL = {J. \'Ec. polytech. Math.},
  FJOURNAL = {Journal de l'\'Ecole polytechnique. Math\'ematiques},
    VOLUME = {12},
      YEAR = {2025},
     PAGES = {983--1018},
      ISSN = {2429-7100,2270-518X},
   MRCLASS = {14G40 (11G50 14J45 32Q20 53C25)},
  MRNUMBER = {4921185},
       DOI = {10.5802/jep.304},
       URL = {https://doi.org/10.5802/jep.304},
}

@article {BS09,
    AUTHOR = {Bauer, Thomas and Szemberg, Tomasz},
     TITLE = {Seshadri constants and the generation of jets},
   JOURNAL = {J. Pure Appl. Algebra},
   VOLUME = {213},
      YEAR = {2009},
    NUMBER = {11},
     PAGES = {2134–2140},
}

@article {Ber25,
    AUTHOR = {Berman, Robert J.},
     TITLE = {On {K}-stability, height bounds and the {M}anin-{P}eyre
              conjecture},
   JOURNAL = {Pure Appl. Math. Q.},
  FJOURNAL = {Pure and Applied Mathematics Quarterly},
    VOLUME = {21},
      YEAR = {2025},
    NUMBER = {3},
     PAGES = {931--970},
      ISSN = {1558-8599,1558-8602},
   MRCLASS = {14G05 (14E30 14J45 32Q20 32Uxx)},
  MRNUMBER = {4852025},
MRREVIEWER = {Ariyan\ Javanpeykar},
       DOI = {10.4310/pamq.250115041532},
       URL = {https://doi.org/10.4310/pamq.250115041532},
}

@article {BDPP13,
    AUTHOR = {Boucksom, S\'ebastien and Demailly, Jean-Pierre and P{\u a}un,
              Mihai and Peternell, Thomas},
     TITLE = {The pseudo-effective cone of a compact {K}\"ahler manifold and
              varieties of negative {K}odaira dimension},
   JOURNAL = {J. Algebraic Geom.},
  FJOURNAL = {Journal of Algebraic Geometry},
    VOLUME = {22},
      YEAR = {2013},
    NUMBER = {2},
     PAGES = {201--248},
      ISSN = {1056-3911,1534-7486},
   MRCLASS = {14E99 (32J18 32L05 53C26)},
  MRNUMBER = {3019449},
MRREVIEWER = {Thomas\ Eckl},
       DOI = {10.1090/S1056-3911-2012-00574-8},
       URL = {https://doi.org/10.1090/S1056-3911-2012-00574-8},
}

@article {Fuj11,
    AUTHOR = {Osamu Fujino},
     TITLE = {Fundamental theorems for the log minimal model program},
   JOURNAL = {Publ. Res. Inst. Math.Sci.},
  FJOURNAL = {},
    VOLUME = {47},
      YEAR = {2011},    
     PAGES = {727-789},
     
}

@article {JLR26, 
    AUTHOR = {Jovinelly, Eric and Lehmann, Brian and Riedl, Eric},
     TITLE = {Optimal bounds in bend-and-break},
   JOURNAL = {Forum of Mathematics, Pi},
  FJOURNAL = {},
    VOLUME = {14:e16},
      YEAR = {2026},    
     PAGES = {arXiv:2509.08065},
}

@article{HM26,
 author = {Honda, Shouhei and Mondino, Andrea},
 title = {Gap phenomena under curvature restrictions},
 fjournal = {Indagationes Mathematicae. New Series},
 journal = {Indag. Math., New Ser.},
 issn = {0019-3577},
 volume = {37},
 number = {3},
 pages = {744--763},
 year = {2026},
 language = {English},
 doi = {10.1016/j.indag.2025.03.008},
 zbMATH = {8196484}
}

@article {CD15,
    AUTHOR = {Casagrande, Cinzia and Druel, St\'{e}phane},
     TITLE = {Locally unsplit families of rational curves of large anticanonical degree on {F}ano manifolds},
   JOURNAL = {Int. Math. Res. Not. IMRN},
  FJOURNAL = {},
    VOLUME = {},
      YEAR = {2015},
    NUMBER = {21},
     PAGES = {10756-10800},
}

@article {CFH14,
    AUTHOR = {Chen, Yifei and Fu, Baohua and Hwang, Jun-Muk},
     TITLE = {Minimal rational curves on complete toric manifolds and
              applications},
   JOURNAL = {Proc. Edinb. Math. Soc. (2)},
  FJOURNAL = {Proceedings of the Edinburgh Mathematical Society. Series II},
    VOLUME = {57},
      YEAR = {2014},
    NUMBER = {1},
     PAGES = {111--123},
      ISSN = {0013-0915,1464-3839},
   MRCLASS = {14M25 (14H10)},
  MRNUMBER = {3165015},
MRREVIEWER = {Nathan\ Owen\ Ilten},
       DOI = {10.1017/S001309151300093X},
       URL = {https://doi.org/10.1017/S001309151300093X},
}

@InCollection{CMSB02,
 Author = {Cho, Koji and Miyaoka, Yoichi and Shepherd-Barron, N.I.},
 Title = {Characterizations of projective space and applications to complex symplectic manifolds},
 BookTitle = {Higher dimensional birational geometry (Kyoto, 1997), Adv. Stud. Pure Math., 25},
 ISBN = {},
 Pages = {1-88},
 Year = {2002},
 Publisher = {Mathematical Society of Japan},
 Language = {},
 zbMATH = {},
 Zbl = {}
}

@article {DH17,
    AUTHOR = {Dedieu, Thomas and H\"{o}ring, Andreas},
     TITLE = {Numerical characterisation of quadrics},
   JOURNAL = {Algebraic Geometry},
  FJOURNAL = {},
    VOLUME = {4},
      YEAR = {2017},
    NUMBER = {1},
     PAGES = {120-135},
      ISSN = {},
   MRCLASS = {},
  MRNUMBER = {},
MRREVIEWER = {},
       DOI = {},
       URL = {},
}

@article {LL19,
    AUTHOR = {Li, Chi and Liu, Yuchen},
     TITLE = {K\"{a}hler-{E}instein metrics and volume minimization.},
   JOURNAL = {Adv. Math.},
  FJOURNAL = {},
    VOLUME = {341},
      YEAR = {2019},
    NUMBER = {},
     PAGES = {440-492},
}

@article {LM25,
    AUTHOR = {Li, Chi and Miao, Minghao},
     TITLE = {On the volume of {K}-semistable {F}ano manifolds},
   JOURNAL = {arXiv preprint arXiv:2506.17420},
      YEAR = {2025},
}

@article {Mor79,
    AUTHOR = {Mori, Shigefumi},
     TITLE = {Projective manifolds with ample tangent bundles},
   JOURNAL = {Ann. of Math. (2)},
  FJOURNAL = {Annals of Mathematics. Second Series},
    VOLUME = {110},
      YEAR = {1979},
    NUMBER = {3},
     PAGES = {593--606},
      ISSN = {0003-486X},
   MRCLASS = {14E05 (14M20)},
  MRNUMBER = {554387},
MRREVIEWER = {Luciana\ Picco Botta},
       DOI = {10.2307/1971241},
       URL = {https://doi.org/10.2307/1971241},
}

@article {MM86,
    AUTHOR = {Miyaoka, Yoichi and Mori, Shigefumi},
     TITLE = {A numerical criterion for uniruledness},
   JOURNAL = {Ann. of Math. (2)},
  FJOURNAL = {Annals of Mathematics. Second Series},
    VOLUME = {124},
      YEAR = {1986},
    NUMBER = {1},
     PAGES = {65--69},
      ISSN = {0003-486X,1939-8980},
   MRCLASS = {14J40},
  MRNUMBER = {847952},
MRREVIEWER = {C.\ A. M. Peters},
       DOI = {10.2307/1971387},
       URL = {https://doi.org/10.2307/1971387},
}

@article {Zhang22,
    AUTHOR = {Zhang, Kewei},
     TITLE = {On the optimal volume upper bound for {K}\"ahler manifolds
              with positive {R}icci curvature (with an appendix by {Y}uchen
              {L}iu)},
   JOURNAL = {Int. Math. Res. Not. IMRN},
  FJOURNAL = {International Mathematics Research Notices. IMRN},
      YEAR = {2022},
    NUMBER = {8},
     PAGES = {6135--6156},
      ISSN = {1073-7928,1687-0247},
   MRCLASS = {53C55 (14M25 32Q15 53C21)},
  MRNUMBER = {4406126},
MRREVIEWER = {Gabriel\ Eduard\ V\^ilcu},
       DOI = {10.1093/imrn/rnaa295},
       URL = {https://doi.org/10.1093/imrn/rnaa295},
}

@article {Zhang21,
    AUTHOR = {Zhang, Kewei},
     TITLE = {Continuity of delta invariants and twisted
              {K}\"ahler-{E}instein metrics},
   JOURNAL = {Adv. Math.},
  FJOURNAL = {Advances in Mathematics},
    VOLUME = {388},
      YEAR = {2021},
     PAGES = {Paper No. 107888, 25},
      ISSN = {0001-8708,1090-2082},
   MRCLASS = {53C55 (14C20 53C25)},
  MRNUMBER = {4288212},
MRREVIEWER = {Yanir\ A.\ Rubinstein},
       DOI = {10.1016/j.aim.2021.107888},
       URL = {https://doi.org/10.1016/j.aim.2021.107888},
}

@article {LM09,
    AUTHOR = {Lazarsfeld, Robert and Musta{\c{t}}{\u{a}}, Mircea},
     TITLE = {Convex bodies associated to linear series},
   JOURNAL = {Ann. Sci. \'{E}c. Norm. Sup\'{e}r. (4)},
  FJOURNAL = {Annales Scientifiques de l'\'{E}cole Normale Sup\'{e}rieure. Quatri\`eme
              S\'{e}rie},
    VOLUME = {42},
      YEAR = {2009},
    NUMBER = {5},
     PAGES = {783--835},
      ISSN = {0012-9593},
   MRCLASS = {14C20 (14E05)},
  MRNUMBER = {2571958},
MRREVIEWER = {Zach Teitler},
       DOI = {10.24033/asens.2109},
       URL = {https://doi.org/10.24033/asens.2109},
}

@article {MS24,
    AUTHOR = {Moraga, Joaqu\'{i}n and S\"{u}ss, Hendrik},
     TITLE = {Bounding toric singularities with normalized volume},
   JOURNAL = {Bull. Lond. Math. Soc.},
  FJOURNAL = {},
    VOLUME = {56},
      YEAR = {2024},
    NUMBER = {6},
     PAGES = {2212-2229},
}

@article {P94,
    AUTHOR = {Perelman, G.},
     TITLE = {Manifolds of positive {R}icci curvature with almost maximal
              volume},
   JOURNAL = {J. Amer. Math. Soc.},
  FJOURNAL = {Journal of the American Mathematical Society},
    VOLUME = {7},
      YEAR = {1994},
    NUMBER = {2},
     PAGES = {299--305},
      ISSN = {0894-0347,1088-6834},
   MRCLASS = {53C21 (53C23)},
  MRNUMBER = {1231690},
MRREVIEWER = {Conrad\ Plaut},
       DOI = {10.2307/2152760},
       URL = {https://doi.org/10.2307/2152760},
}

@article {CC97,
    AUTHOR = {Cheeger, Jeff and Colding, Tobias H.},
     TITLE = {On the structure of spaces with {R}icci curvature bounded
              below. {I}},
   JOURNAL = {J. Differential Geom.},
  FJOURNAL = {Journal of Differential Geometry},
    VOLUME = {46},
      YEAR = {1997},
    NUMBER = {3},
     PAGES = {406--480},
      ISSN = {0022-040X,1945-743X},
   MRCLASS = {53C21 (53C20)},
  MRNUMBER = {1484888},
MRREVIEWER = {William\ P.\ Minicozzi, II},
       URL = {http://projecteuclid.org/euclid.jdg/1214459974},
}

@article {KL17,
    AUTHOR = {K\"uronya, Alex and Lozovanu, Victor},
     TITLE = {Positivity of line bundles and {N}ewton-{O}kounkov bodies},
   JOURNAL = {Doc. Math.},
  FJOURNAL = {Documenta Mathematica},
    VOLUME = {22},
      YEAR = {2017},
     PAGES = {1285--1302},
      ISSN = {1431-0635,1431-0643},
   MRCLASS = {14C20 (32L20 32Q99)},
  MRNUMBER = {3722563},
MRREVIEWER = {Halszka\ Tutaj-Gasi\'nska},
}

@article{B02,
 author = {Boucksom, S{\'e}bastien},
 title = {On the volume of a line bundle.},
 fjournal = {International Journal of Mathematics},
 journal = {Int. J. Math.},
 issn = {0129-167X},
 volume = {13},
 number = {10},
 pages = {1043--1063},
 year = {2002},
 language = {English},
 doi = {10.1142/S0129167X02001575},
 zbMATH = {1911923},
 Zbl = {1101.14008}
}

@misc{BM87,
 author = {Bando, Shigetoshi and Mabuchi, Toshiki},
 title = {Uniqueness of {Einstein} {K{\"a}hler} metrics modulo connected group actions},
 year = {1987},
 language = {English},
 howpublished = {Algebraic geometry, {Proc}. {Symp}., {Sendai}/{Jap}. 1985, {Adv}. {Stud}. {Pure} {Math}. 10, 11-40 (1987).},
 zbMATH = {4044682},
 Zbl = {0641.53065}
}

@article{BBJ,
 author = {Berman, Robert J. and Boucksom, S{\'e}bastien and Jonsson, Mattias},
 title = {A variational approach to the {Yau}-{Tian}-{Donaldson} conjecture},
 fjournal = {Journal of the American Mathematical Society},
 journal = {J. Am. Math. Soc.},
 issn = {0894-0347},
 volume = {34},
 number = {3},
 pages = {605--652},
 year = {2021},
 language = {English},
 doi = {10.1090/jams/964},
 zbMATH = {7420176},
 Zbl = {1487.32141}
}

@article{CRZ,
 author = {Cheltsov, Ivan A. and Rubinstein, Yanir A. and Zhang, Kewei},
 title = {Basis log canonical thresholds, local intersection estimates, and asymptotically log del {Pezzo} surfaces},
 fjournal = {Selecta Mathematica. New Series},
 journal = {Sel. Math., New Ser.},
 issn = {1022-1824},
 volume = {25},
 number = {2},
 pages = {36},
 note = {Id/No 34},
 year = {2019},
 language = {English},
 doi = {10.1007/s00029-019-0473-z},
 zbMATH = {7056270},
 Zbl = {1418.32015}
}

@article{FO18,
 author = {Fujita, Kento and Odaka, Yuji},
 title = {On the {K}-stability of {Fano} varieties and anticanonical divisors},
 fjournal = {T{\^o}hoku Mathematical Journal. Second Series},
 journal = {T{\^o}hoku Math. J. (2)},
 issn = {0040-8735},
 volume = {70},
 number = {4},
 pages = {511--521},
 year = {2018},
 language = {English},
 doi = {10.2748/tmj/1546570823},
 zbMATH = {7040974},
 Zbl = {1422.14047}
}

@article{BJ20,
 author = {Blum, Harold and Jonsson, Mattias},
 title = {Thresholds, valuations, and {K}-stability},
 fjournal = {Advances in Mathematics},
 journal = {Adv. Math.},
 issn = {0001-8708},
 volume = {365},
 pages = {57},
 note = {Id/No 107062},
 year = {2020},
 language = {English},
 doi = {10.1016/j.aim.2020.107062},
 zbMATH = {7184864},
 Zbl = {1441.14137}
}

@book{K96,
 author = {Koll{\'a}r, J{\'a}nos},
 title = {Rational curves on algebraic varieties},
 fseries = {Ergebnisse der Mathematik und ihrer Grenzgebiete. 3. Folge},
 series = {Ergeb. Math. Grenzgeb., 3. Folge},
 issn = {0071-1136},
 volume = {32},
 isbn = {3-540-60168-6},
 year = {1996},
 publisher = {Berlin: Springer-Verlag},
 language = {English},
 zbMATH = {833161},
 Zbl = {0877.14012}
}

@article{R08,
 author = {Rubinstein, Yanir A.},
 title = {Some discretizations of geometric evolution equations and the {Ricci} iteration on the space of {K{\"a}hler} metrics},
 fjournal = {Advances in Mathematics},
 journal = {Adv. Math.},
 issn = {0001-8708},
 volume = {218},
 number = {5},
 pages = {1526--1565},
 year = {2008},
 language = {English},
 doi = {10.1016/j.aim.2008.03.017},
 zbMATH = {5299371},
 Zbl = {1143.53065}
}

@article{T92,
 author = {Tian, Gang},
 title = {On stability of the tangent bundles of {Fano} varieties},
 fjournal = {International Journal of Mathematics},
 journal = {Int. J. Math.},
 issn = {0129-167X},
 volume = {3},
 number = {3},
 pages = {401--413},
 year = {1992},
 language = {English},
 doi = {10.1142/S0129167X92000175},
 zbMATH = {90946},
 Zbl = {0779.53009}
}

@article{Sz11,
 author = {Sz{\'e}kelyhidi, G{\'a}bor},
 title = {Greatest lower bounds on the {Ricci} curvature of {Fano} manifolds},
 fjournal = {Compositio Mathematica},
 journal = {Compos. Math.},
 issn = {0010-437X},
 volume = {147},
 number = {1},
 pages = {319--331},
 year = {2011},
 language = {English},
 doi = {10.1112/S0010437X10004938},
 zbMATH = {5854937},
 Zbl = {1222.32046}
}

@article{SW16,
 author = {Song, Jian and Wang, Xiaowei},
 title = {The greatest {Ricci} lower bound, conical {Einstein} metrics and {Chern} number inequality},
 fjournal = {Geometry \& Topology},
 journal = {Geom. Topol.},
 issn = {1465-3060},
 volume = {20},
 number = {1},
 pages = {49--102},
 year = {2016},
 language = {English},
 doi = {10.2140/gt.2016.20.49},
 zbMATH = {6553428},
 Zbl = {1342.32018}
}

@article{Li11,
 author = {Li, Chi},
 title = {Greatest lower bounds on {Ricci} curvature for toric {Fano} manifolds},
 fjournal = {Advances in Mathematics},
 journal = {Adv. Math.},
 issn = {0001-8708},
 volume = {226},
 number = {6},
 pages = {4921--4932},
 year = {2011},
 language = {English},
 doi = {10.1016/j.aim.2010.12.023},
 zbMATH = {5871176},
 Zbl = {1222.14090}
}

@book{Lbook,
 author = {Lazarsfeld, Robert},
 title = {Positivity in algebraic geometry. {I}. {Classical} setting: line bundles and linear series},
 fseries = {Ergebnisse der Mathematik und ihrer Grenzgebiete. 3. Folge},
 series = {Ergeb. Math. Grenzgeb., 3. Folge},
 issn = {0071-1136},
 volume = {48},
 isbn = {3-540-22533-1},
 year = {2004},
 publisher = {Berlin: Springer},
 language = {English},
 zbMATH = {2134816},
 Zbl = {1093.14501}
}

@article{LZ18,
 author = {Liu, Yuchen and Zhuang, Ziquan},
 title = {Characterization of projective spaces by {Seshadri} constants},
 fjournal = {Mathematische Zeitschrift},
 journal = {Math. Z.},
 issn = {0025-5874},
 volume = {289},
 number = {1-2},
 pages = {25--38},
 year = {2018},
 language = {English},
 doi = {10.1007/s00209-017-1941-9},
 zbMATH = {6898007},
 Zbl = {1391.14081}
}

@article{Li17,
 author = {Li, Chi},
 title = {Yau-{Tian}-{Donaldson} correspondence for {K}-semistable {Fano} manifolds},
 fjournal = {Journal f{\"u}r die Reine und Angewandte Mathematik},
 journal = {J. Reine Angew. Math.},
 issn = {0075-4102},
 volume = {733},
 pages = {55--85},
 year = {2017},
 language = {English},
 doi = {10.1515/crelle-2014-0156},
 zbMATH = {6815198},
 Zbl = {1388.53076}
}

@article{Fvol,
 author = {Fujita, Kento},
 title = {Optimal bounds for the volumes of {K{\"a}hler}-{Einstein} {Fano} manifolds},
 fjournal = {American Journal of Mathematics},
 journal = {Am. J. Math.},
 issn = {0002-9327},
 volume = {140},
 number = {2},
 pages = {391--414},
 year = {2018},
 language = {English},
 doi = {10.1353/ajm.2018.0009},
 zbMATH = {6924827},
 Zbl = {1400.14105}
}
\Addresses
\end{document}